\documentclass{amsart}
\usepackage{marktext}
\usepackage{gimac}

\usepackage[tight]{subfigure}

\colorlet{darkred}{red!40!black}

\newdelim{\ip}{\langle}{\rangle}

\makeatletter
\newcommand{\leqnomode}{\tagsleft@true\let\veqno\@@leqno}
\newcommand{\reqnomode}{\tagsleft@false\let\veqno\@@eqno}
\newcommand{\mylabel}[2]{\def\@currentlabel{#2}\label{#1}}
\makeatother

\newcommand{\FF}{\mathcal{F}}
\newcommand{\FL}{\mathcal{L}}
\newcommand{\FS}{\mathcal{S}}
\newcommand{\calT}{\mathcal{T}}

\def \k {\mathbf k}

\begin{document}
\title[Suppression of epitaxial thin film growth by mixing]%
  {Suppression of epitaxial thin film growth by mixing}
\author[Feng]{Yu Feng}
\address{%
  Yu Feng, Department of Mathematics, University of Wisconsin --
  Madison, Madison, WI 53706}
\email{feng65@wisc.edu}

\author[Hu]{Bingyang Hu}
\address{%
  Bingyang Hu, Department of Mathematics, Purdue University, 150 N. University St.,W.Lafayette, IN 47907}
\email{hu776@purdue.edu}

\author[Xu]{Xiaoqian Xu}
\address{%
  Xiaoqian Xu, Department of Mathematics, Duke Kunshan University, 8 Duke Ave, Kunshan, Suzhou, Jiangsu, China, 215316}
\email{xiaoqian.xu@dukekunshan.edu.cn}

\date{\today}

\begin{abstract}
We consider following fourth-order parabolic equation with gradient nonlinearity on the two-dimensional torus with and without advection of an incompressible vector field in the case $2<p<3$:
\begin{equation*}
    \partial_t u + (-\lap)^2 u = -\grad\cdot(\abs{\grad u}^{p-2}\grad u).
\end{equation*}
The study of this form of equations arises from mathematical models that simulate the epitaxial growth of the thin film. We prove the local existence of mild solutions for any initial data lies in $L^2$ in both cases. Our main result is: in the advective case, if the imposed advection is sufficiently mixing, then the global existence of solution can be proved, and the solution will converge exponentially to a homogeneous mixed state. While in the absence of advection, there exist initial data in $H^2\cap W^{1,\infty}$ such that the solution will blow up in finite time.
\end{abstract}
\maketitle

\section{Introduction}
The global well-posedness and finite-time blow-up of the solution to a nonlinear parabolic PDE is of widespread interest and arises in different applications in many different areas, especially in mathematical biology and fluid dynamics. Advection is often discussed in the corresponding papers, and the presence of it sometimes may stabilize the singularity. We refer the interested reader  \cites{FannjiangKiselevEA06,BerestyckiKiselevEA10,KiselevXu16,BedrossianHe17,He18,HeTadmor19,feng2020global} and the references therein for more details. As a remark, the stabilization of the fluid flow was proposed to study Keller-Segel type equations with some constraints. Later, in \cites{feng2020phase,feng2020global}, the authors considered such a phenomenon for more general types of equations such as Cahn-Hilliard equations and Kuramoto-Sivashinsky equations. One common ingredient for most of these previous results is to explore the special convective motions with mixing flows, whose dissipation-enhancing properties have been understood well, and we refer the interested reader \cites{ConstantinKiselevEA08, YaoZlatos17, AlbertiCrippaEA16, ElgindiZlatos19, Zlatos10,CotiZelatiDelgadinoEA18,FengIyer19, iyer2019convection} for more examples about such flows.

In this article, we consider the following fourth-order parabolic equations with gradient nonlinearity, which presents one of the continuum models for epitaxial thin film growth (see, e.g., \cite{ortiz1999continuum,zangwill1996some}): for any $p>2$,   
\begin{equation} \label{20200716eq01}
    h_t + A_1\lap h + A_2\lap^2 h + A_3\grad\cdot(\abs{\grad h}^{p-2}\grad h) = g, 
\end{equation}
where $h(t, x)$ denotes the height of a film in epitaxial growth with $g(t,x)$ being the deposition flux and $A_1, A_2, A_3 \in \R$. The spatial derivatives in the above equation have the following physical interpretations:
\begin{itemize}
    \item [(1). ] $A_1\lap h$: diffusion due to evaporation-condensation~\cites{edwards1982surface,mullins1957theory}; 
    \item [(2). ] $A_2\lap^2 h$: capillarity-driven surface diffusion~\cites{herring1999surface,mullins1957theory}; 
    \item [(3). ] $A_3\grad\cdot(\abs{\grad h}^{p-2}\grad h)$: (upward) hopping of atoms~\cite{sarma1992solid}.
\end{itemize}
These models simulate the complex process of making a thin film layer on a substrate by chemical vapor deposition, and one of the most interesting questions is to understand these growth processes quantitatively on the correct scale so that people can optimize the particular properties of the film. These thin-film type equations have been studied by many authors (see, e.g., \cites{ishige2020blowup,king2003fourth,sandjo2014space,li2003thin,sandjo2015solutions}). One can also refer to \cites{ortiz1999continuum,schulze1999geometric} for more development of such equations, which concentrates on the background of material science. 

In this paper, we will concentrate on the analytical properties of the model \eqref{20200716eq01} on two dimensional torus~$\T^2=\left[0,1 \right]^2$. One typical situation for such models is the case when $A_1=0$, $A_2=A_3=1$ and $g=0$:
\begin{equation}
\label{e:4 order eqn}
    \partial_t u + (-\lap)^2 u = -\grad\cdot(\abs{\grad u}^{p-2}\grad u), \qquad 
    u(0,x)= u_0(x), 
\end{equation}
where $p>2$ and $\partial_t\defeq\partial/\partial t$. For simplicity, we denote $N(u)\defeq \grad\cdot F(\grad u)$, where $F(\xi)=\abs{\xi}^{p-2}\xi$. It is worth to mention that~\eqref{e:4 order eqn} can be regarded as the $L^2$-gradient flow for the energy functional
\begin{equation*}
    E(\phi)\defeq 
    \frac{1}{2}\int_{\T^d}(\lap\phi)^2\,dx - \frac{1}{p}\int_{\T^d}\abs{\grad\phi}^{p}\,dx
    = \frac{1}{2}\norm{\lap\phi}_2^2 - \frac{1}{p}\norm{\grad\phi}_p^p,
\end{equation*}
which satisfies
\begin{equation*}
    \lim_{\eta\rightarrow\infty} E(\eta\phi)=-\infty
\end{equation*}
for $\phi\in H^2(\T^2)\cap W^{1,p}(\T^2)\setminus\{0\}$. For the nonlinear term in \eqref{e:4 order eqn}, the local existence and singularity behavior was first studied in \cite{ishige2020blowup} on $\mathbb{R}^N$.
In this paper, we study~\eqref{e:4 order eqn} with and without advection is given by an incompressible vector field. We will prove that the advection term with good mixing property would prevent the finite time blowup for non-classical solutions to this equation. In this article, we focus on the case $2<p<3$ on $\T^2$, however, a parallel argument can extend our results to the three-dimension torus with a smaller range of $p$, we leave the details to the interested reader.

Throughout this paper, $C$ denotes a generic constant that may change from line to line.

\medskip

This paper is organized as follows: Section \ref{fengyuxiaomeimei} is devoted to studying the local well-posedness of the solution to \eqref{e:4 order eqn}; in Section \ref{fengyuxiaogege}, we show that adding advection term to \eqref{e:4 order eqn} may enable the existence of global solution if we choose the velocity field carefully with respect to the initial data; finally, in Section~\ref{sec:final sec} we give examples of the blow-up of the solutions, and also characterize the blow-up behavior by providing a quantitative blow-up rate of the $L^2$ norm.

\subsection*{Acknowledgments}
The authors would like to thank Yuanyuan Feng for useful discussions. We would also like to thank the anonymous referees for the helpful comments.

\section{Epitaxial thin film growth equation without advection} \label{fengyuxiaomeimei}
In this section we study~\eqref{e:4 order eqn} without advection on $\T^2$ under the regime $2<p<3$. We begin with introducing some notation that will be used throughout this paper. Given a function $f\in L^{p}(\T^2)$ with $p\geq 1$, we denote $\hat{f}(\k)$ to be the Fourier coefficient of $f$ at frequency $\k\in\mathbb{Z}^2$.
Now for $s \in \mathbb{R}_{+}$, we can define the Sobolev space $H^{s}(\T^2)$ and the homogeneous Sobolev space $\dot{H}^{s}(\T^2)$  by the collection of measurable functions $f$ on $\T^2$, with 
$$
    \norm{f}_{H^s}^2 \defeq \sum_{\k\in\mathbb{Z}^2}(1+\abs{\k}^2)^s\abs{\hat{f}(\k)}^2=\norm{(I-\lap)^{s/2}f}_{L^2}^2
$$
and
$$
\norm{f}_{\dot{H}^s}^2 \defeq \sum_{\k\in\mathbb{Z}^2}\abs{\k}^{2s}\abs{\hat{f}(\k)}^2=\norm{(-\lap)^{s/2}f}_{L^2}^2, 
$$
respectively, where $(-\lap)^{s/2}$ agrees with the Fourier multiplier with symbol $\abs{\k}^s$, $\k\neq0$ and $I$ is the identity operator. Note that for $f \in L^2(\T^2)$, $f \in H^s(\T^2)$ if and only if $f \in \dot{H}^s(\T^2)$.

Let $\mathcal{L}=\lap^2$, which maps $H^4$ to $L^2$, and $e^{-t\FL}$ be the strongly continuous semigroup generated by $\mathcal{L}$ on $L^2$ given by
\begin{equation*}
    e^{-t\FL} f \defeq \FF^{-1}\Big(e^{-t\abs{\k}^4}\hat{f}\Big),
\end{equation*}
where $\FF^{-1}$ denotes the inverse Fourier transform on $\Z^2$. For the reader's convenience, we will recall some useful properties of $e^{-t \FL}$ in Lemma~\ref{lem: semi group on divergence} and Lemma~\ref{lem: Hs of semigroup}. To this end, for any measurable function $\phi(t, x)$ defined on $\R_{\ge 0} \times \T^2$, we denote $\phi(t)$ to be the function $\phi(t, \cdot)$, that is, for any $x \in \T^2$, $\phi(t)(x)=\phi(t,x)$. Next we recall the notion of mild and weak solutions in the following definitions.

\begin{definition}
\label{def:mild and weak soln}
\begin{enumerate}
\item [(1).] For $p>2$, a function $ u\in C([0,T];L^{2}(\mathbb{T}^2)),T>0$, such that $\grad u$ is locally integrable, is called a \emph{mild solution} of~\eqref{e:4 order eqn} on $[0,T]$ with initial data $ u_0\in L^{2}(\T^2)$, if for any $0 \le t \le T$, 
\begin{equation}
\label{e:volterra integral 1}
     u(t)=\mathcal{T}(u)(t)\defeq e^{-t\FL} u_0-\int_0^t\, \grad e^{-(t-s)\FL}\Big(\abs{\grad u}^{p-2}\grad u\Big)\,ds
\end{equation}
holds pointwisely in time with values in $L^2$, where the integral is defined in the B\"{o}chner sense. 

\item [(2).] For $p>2$, a function $ u\in L^{\infty}([0,T];L^{2}(\T^{2}))\cap L^2([0,T];H^2(\T^{2}))$ is called a \emph{weak solution} of~\eqref{e:4 order eqn} on $[0,T)$ with initial data $u_0\in L^{2}(\T^{2})$ if for all $\phi\in C_c^{\infty}([0,T)\times\T^{2})$,
\begin{align}
\label{e:weak soln}
&\int_{\T^2}\,u_0\phi(0)\,dx\, +\, \int_0^T\,\int_{\T^2}u\,\partial_t \phi\,dx dt \notag \\ 
&=\int_0^T\,\int_{\T^2}\,\lap u\lap\phi dx dt\, -\, \int_0^T \int_{\T^2}\,\abs{\grad u}^{p-2}\grad u\cdot\grad\phi\,dx dt,
\end{align}
and $\partial_t u\in L^{2}([0,T]; H^{-2}(\T^2))$. 
\end{enumerate}
\end{definition}
Mild solutions are formally fixed points of the non-linear map $\mathcal{T}$, and~\eqref{e:volterra integral 1} is in the form of a Volterra integral equation.
\begin{remark}
The mild solution with rough initial data ($L^2$) established above is quite different from the ones considered in the previous studies (see ~\cites{ishige2020blowup,sandjo2014space,sandjo2015solutions}). In Section~\ref{fengyuxiaogege}, we show that this type of mild solution can be easily extended to be a global one by adding an advection term with specific mixing property compared to the mild solutions considered before.
\end{remark}

\subsection{Local existence with $L^2$ initial data}
When $3<p<4$, the local existence of the mild solution to \eqref{e:4 order eqn} on $\mathbb{R}^N$ was studied in \cites{sandjo2014space,sandjo2015solutions}. In the recent paper \cite{ishige2020blowup}, the authors provided the local existence result of the mild solution for a better range $2<p\leq 4$ on $\mathbb{R}^N$ with initial data that has better regularity (rather than $L^2(\T^2)$). While in this paper, we focus on the case $2<p< 3$ on $\mathbb{T}^2$ with rough initial data. We show for any initial data $u_0\in L^2(\T^2)$, there exists $0<T<1$ depends on $\norm{u_0}_{L^2}$ and $p$ such that~\eqref{e:4 order eqn} admits a mild solution on $[0,T]$. Moreover, in Proposition~\ref{prop:mild equi with weak}, we show that this mild solution is also a weak solution to~\eqref{e:4 order eqn}. Finally, it's worth mentioning that the existence of weak solution and classical solution to~\eqref{e:4 order eqn} for $p>2$ was studied in~\cite{king2003fourth}.

A standard way to prove the local existence of a mild solution is to apply the Banach contraction mapping theorem, hence it suffices to argue that  $\mathcal{T}$ is a contraction map in a suitable adapted Banach space $\Tilde{\FS}_T$,  which is defined as follows. 

Given $0<T<1$, we let
\begin{equation*}
    \FS_T\defeq\{ u:\mathbb{R}_+\times\T^2\rightarrow\mathbb{R}\vert\sup_{0<t\leq T}t^{1/4}\norm{\grad u}_{L^2}<\infty\},
\end{equation*}
and 
\begin{equation*}
    \Tilde{\FS}_T\defeq C([0,T];L^2(\T^2))\cap\FS_T,
\end{equation*}
It is clear that $\Tilde{\FS}_T$ is a Banach space equipped with the norm:
\begin{equation*}
    \norm{ u}_{\Tilde{\FS}_T} \defeq \max\left\{\sup_{0\leq t\leq T}\norm{ u}_{L^2},\sup_{0< t\leq T}t^{1/4}\norm{\grad u}_{L^{2}}\right\}.
\end{equation*}
For completeness, we first recall and prove two useful estimates for the semigroup operator $e^{-t\FL}$ in the following two lemmas. Then we use these estimates to verify that $\mathcal{T}$ is a contraction map on a ball in $\Tilde{\FS}_T$.

\begin{lemma}
\label{lem: semi group on divergence}
For $2<p \leq 3$, there exists a constant $C$ such that 
\begin{equation*}
    \norm{e^{-t\FL}f}_{L^2}\leq Ct^{-\frac{p-2}{4}}\norm{f}_{L^\frac{2}{p-1}}.
\end{equation*}
\end{lemma}
\begin{proof}
We begin with the case when $2<p<3$. By definition of operator $e^{-t\FL}$ and Plancherel's identity, we can easily get:
\begin{align*}
   \norm{e^{-t\FL}f}_{L^2}^2 
   &=\sum_{\k\in\mathbb{Z}^2}e^{-2t\abs{\k}^4} \abs{\hat{f}(\k)}^2
   \leq \left(\sum_{k \in \Z^2} e^{-\frac{2t|\k|^4}{p-2}} \right)^{p-2} \left(\sum_{k \in \Z^2} \left|\hat{f}(\k) \right|^{\frac{2}{3-p}} \right)^{3-p} \\
   &\leq C\norm{f}_{L^\frac{2}{p-2}}^2\Big(\int_{\mathbb{R}^2}\, e^{-\frac{2t|x|^4}{p-2}} dx\Big)\leq C t^{-\frac{p-2}{2}} \norm{f}_{L^\frac{2}{p-2}}^2,
\end{align*}
where in the second last estimate above, we have used the assumption $p>2$ and the Hausdorff–Young inequality on $\T^2$. While for the case when $p=3$, we simply bound $\|e^{-t\FL} f\|_{L^2}^2$ by 
$$
\left\|\hat{f} \right\|^2_{\ell^\infty(\Z^2)} \cdot \sum_{\k \in \Z^2} e^{-2t|\k|^4}.
$$
The rest of the proof follows in a similar fashion and hence we omit it here. 
\end{proof}

\begin{lemma}
\label{lem: Hs of semigroup}
For any $s>0$, there exists a constant $C$ such that
\begin{equation*}
    \norm{(-\lap)^{s/2}e^{-t\FL}f}_{L^2}\leq Ct^{-\frac{s}{4}}\norm{f}_{L^2}.
\end{equation*}
\end{lemma}
\begin{proof}
The proof of Lemma \ref{lem: Hs of semigroup} is similar to the one of Lemma \ref{lem: semi group on divergence} and hence we would like to leave the detail to the interested reader. 
\end{proof}

The local existence of mild solution is summarized in the following main theorem.
\begin{theorem}
\label{thm:local mild soln}
Let $u_0\in L^{2}(\T^2)$ and $2<p <3$. Then there exists $0< T\leq 1$ depending only on $\norm{u_0}_{L^2}$ such that~\eqref{e:4 order eqn} admits a mild solution $u$ on $[0,T]$, which is unique in $\Tilde{\FS}_T$.
\end{theorem}
\begin{corollary}
\label{cor:L2 argument}
With the same assumptions in Theorem~\ref{thm:local mild soln}, if $T^*$ is the maximal time of existence of the mild solution $u$, then 
\begin{equation*}
    \limsup_{t\rightarrow T_-^*}\norm{u(t)}_{L^2(\T^2)}=\infty.
\end{equation*}
Otherwise, $T^*=\infty$.
\end{corollary}

We divide the proof of main Theorem \ref{thm:local mild soln} into Lemma~\ref{lem:bdd of operator} and Lemma~\ref{lem:L-cts of operator}. In Lemma~\ref{lem:bdd of operator}, we show that the solution operator $\mathcal{T}$ is bounded. In Lemma~\ref{lem:L-cts of operator}, we show the map $\mathcal{T}$ is Lipshitz continuous on $\Tilde{\FS}_T$ with some constant that depends on $T$. Which state as follows:
\begin{lemma}
\label{lem:bdd of operator}
When $2<p < 3$ and $0<T\leq 1$, the map $\mathcal{T}$ maps from $ \Tilde{\FS}_T$ to itself, and there exists $C_1>0$, such that
\begin{equation}
\label{e:bdd of operator}
    \norm{\mathcal{T}( u)}_{\Tilde{\FS}_T}\leqslant C_1\Big(\norm{ u_0}_{L^2}+T^{\frac{3-p}{2}}\norm{ u}_{\Tilde{\FS}_T}^{p-1}\Big).
\end{equation}
\end{lemma}
\begin{lemma}
\label{lem:L-cts of operator}
When $2<p < 3$ and $0<T\leq 1$, there exists a constant $C_2$ such that, for any $ u_1$, $ u_2\in\Tilde{\FS}_T$,
\begin{equation}
\label{e:L-cts of operator}
    \norm{\mathcal{T}( u_1)-\mathcal{T}( u_2)}_{\Tilde{\FS}_T}
    \leq C_2\,T^{\frac{3-p}{2}}\Big(\norm{ u_1}_{\Tilde{\FS}_T}^{p-2}+\norm{ u_2}_{\Tilde{\FS}_T}^{p-2}\Big)\norm{ u_1- u_2}_{\Tilde{\FS}_T}.
\end{equation}
\end{lemma}
Assuming Lemma~\ref{lem:bdd of operator} and Lemma~\ref{lem:L-cts of operator}, we prove Theorem~\ref{thm:local mild soln} first, then use it to show Corollary~\ref{cor:L2 argument}.

\begin{proof}
[Proof of Theorem \ref{thm:local mild soln}]
Let $\mathbb{B}_{R}(0)$ denotes the closed ball centered at origin with radius $R$ in $\Tilde{\FS}_T$ space. Choose $R\geq 2 C_0\norm{u_0}_{L^2}$, where $C_0=\max\{1, C_1,C_2\}$, and $C_1, C_2$ are the constants in~\eqref{e:bdd of operator} and~\eqref{e:L-cts of operator}, and we assume that 
\begin{equation*}
    T\leq \min\Big\{1, (4C_0 R^{p-2})^{\frac{-2}{3-p}}\Big\}.
\end{equation*}
Then, Lemma~\ref{lem:bdd of operator} implies
\begin{equation*}
    \norm{\mathcal{T}(u)}_{\Tilde{\FS}_T}\leq R,\qquad\forall u\in\mathbb{B}_{R}(0).
\end{equation*}
On the other hand, Lemma~\ref{lem:L-cts of operator} yields for any $u_1,u_2\in\mathbb{B}_{R}(0)$ we have
\begin{equation*}
    \norm{\mathcal{T}(u_1)-\mathcal{T}(u_2)}_{\Tilde{\FS}_T}
    \leq \frac{\norm{u_1-u_2}_{\Tilde{\FS}_T}}{2}. 
\end{equation*}
By Banach contraction mapping theorem, there is a unique fixed point of $\mathcal{T}$ in $\mathbb{B}_{R}(0)$. By Definition~\ref{def:mild and weak soln}, $u$ is a mild solution to~\eqref{e:4 order eqn} with initial data $u_0$, which is also unique in $\Tilde{\FS}_T$.
\end{proof}

\begin{proof}[Proof of Corollary \ref{cor:L2 argument}]
We prove the result via contradiction. Assume $T^*<\infty$ and $\limsup\limits_{t\rightarrow T_-^*}\norm{u(t)}_{L^2(\T^2)}=\alpha<\infty$. By continuity, there exists $t^*$ such that for $\forall t^*<t<T^*$, we have $\norm{u(t)}_{L^2(\T^2)}\leq 3\alpha$. Take
\begin{equation*}
     T\leq \min\Big\{1, (4C_0 (3\alpha)^{p-2})^{\frac{-2}{3-p}}\Big\},
\end{equation*}
where $C_0$ is defined in the proof of Theorem~\ref{thm:local mild soln}. Now we pick up any $t_0$ belonging to $[t^*,T^*]$ s.t. $t_0+T\geq T^*$. Then by Theorem~\ref{thm:local mild soln}, there exists a mild solution $\bar{u}$ on $[t_0,t_0+T]$ with initial data $u(t_0)$ and $u=\bar{u}$ on $[t_0,T^*)$, according to the uniqueness of mild solution. Hence, the solution can be extended past $T^*$, which yields a contradiction.
\end{proof}
It now remains to check the lemmas. 

\begin{proof}[Proof of Lemma \ref{lem:bdd of operator}]
To prove Lemma~\ref{lem:bdd of operator}, it suffices to check following two claims:

\medskip

\emph{Claim 1}: If $ u\in\Tilde{\FS}_T$, then $\mathcal{T}( u)\in C([0,T];L^2(\T^2))$;

\medskip

\emph{Claim 2}: If $ u\in\Tilde{\FS}_T$, then $\sup\limits_{0<t\leq T} t^{1/4}\norm{\grad(\mathcal{T}( u))}_{L^2}<\infty$.

\medskip

Utilizing the estimates in Lemma~\ref{lem: semi group on divergence} and Lemma~\ref{lem: Hs of semigroup}, and the assumption $2<p \le 3$, we check \emph{Claim 1} first:
\begin{eqnarray*}
\|\calT(u)\|_{L^2}%
&\leq& \|u_0\|_{L^2}+\int_0^t \left\|\nabla e^{-(t-s)\FL} \left(|\nabla u|^{p-2} \nabla u \right) \right\|_{L^2} ds  \\
&\leq& \|u_0\|_{L^2}+\int_0^t \left\|\nabla e^{-\frac{(t-s)\FL}{2}} \right\|_{L^2 \to L^2} \left\|e^{-\frac{(t-s)\FL}{2}} \left(|\nabla u|^{p-2} \nabla u \right) \right\|_{L^2} ds  \\
&\leq& C \left(\|u_0\|_{L^2}+\int_0^t (t-s)^{-\frac{1}{4}} \cdot (t-s)^{-\frac{p-2}{4}}\|\nabla u\|_{L^2}^{p-1} ds \right) \\
&\leq&  C\Big(\norm{ u_0}_{L^2}+\int_0^t\,(t-s)^{-\frac{p-1}{4}}s^{-\frac{p-1}{4}}\big(s^{\frac{1}{4}}\norm{\grad u}_{L^2}\big)^{p-1}\,ds\Big)\\
&\leq&  C\Big(\norm{ u_0}_{L^2} + t^{\frac{3-p}{2}}\int_0^{1}\,(1-\tau)^{-\frac{1}{2}}\tau^{-\frac{p-1}{4}} \norm{ u}_{\Tilde{\FS}_T}^{p-1}\,d\tau\Big)\\
&\leq&  C\Big(\norm{ u_0}_{L^2}+t^{\frac{3-p}{2}}\norm{ u}_{\Tilde{\FS}_T}^{p-1}\Big),
\end{eqnarray*}
Similarly, for \emph{Claim 2}, we have 
\begin{eqnarray*}
\left\|\nabla(\calT(u))(t) \right\|_{L^2}%
&\leq& \left\|\nabla e^{-t\FL} u_0\right\|_{L^2} \\
&& \quad \quad  + \int_0^t \left\|\lap e^{-\frac{(t-s)\FL}{2}} \right\|_{L^2 \to L^2} \left\|e^{-\frac{(t-s)\FL}{2}} \left(|\nabla u|^{p-2} \nabla u \right) \right\|_{L^2} ds \\
&\leq& C \left( t^{-\frac{1}{4}}\|u_0\|_{L^2}+\int_0^t (t-s)^{-\frac{1}{2}} \cdot  (t-s)^{-\frac{p-2}{4}}\|\nabla u\|_{L^2}^{p-1} ds \right) \\
&=& C \left( t^{-\frac{1}{4}} \|u_0\|_{L^2}+\int_0^t (t-s)^{-\frac{p}{4}} s^{-\frac{p-1}{4}}\big(s^{\frac{1}{4}}\norm{\grad u}_{L^2}\big)^{p-1}\,ds\right)\\
&\leq& C \left(t^{-\frac{1}{4}} \|u_0\|_{L^2}+t^{\frac{5-2p}{4}} \int_0^1 (1-\tau)^{-\frac{p}{4}} \tau^{-\frac{p-1}{4}} d\tau \cdot \|u\|_{\Tilde{\FS}_T}^{p-1} \right) \\
&\leq& C \left(t^{-\frac{1}{4}} \|u_0\|_{L^2}+t^{\frac{5-2p}{4}} \|u\|_{\Tilde{\FS}_T}^{p-1} \right)
\end{eqnarray*}
Multiplying $t^{\frac{1}{4}}$ on both sides of the above estimate, we get
\begin{equation*}
     t^{\frac{1}{4}} \norm{\grad(\mathcal{T}(u)(t))}_{L^2} 
    \leq C\Big(\norm{ u_0}_{L^2}+t^{\frac{3-p}{2}}\norm{ u}_{\Tilde{\FS}_T}^{p-1}\Big),
\end{equation*}
for any $t\in[0,T]$. Combining the proof of two claims, the proof of Lemma~\ref{lem:bdd of operator} is completed. 
\end{proof}
Next, we check that the map $\mathcal{T}$ is Lipschitz on $\Tilde{\FS}_T$ with a constant that depends on $T$.

\begin{proof}[Proof of Lemma \ref{lem:L-cts of operator}]
Similar to the proof of previous lemma, it is sufficient for us to check the following two claims. Let $ u_1, u_2\in\Tilde{\FS}_T$. For all $0\leq t\leq T$:

\emph{Claim 1}.
\begin{equation*}
    \norm{\mathcal{T}(u_1)(t)-\mathcal{T}( u_2)(t)}_{L^2}
    \leqslant C t^{\frac{3-p}{2}}\big(\norm{ u_1}_{\Tilde{\FS}_T}^{p-2}+\norm{ u_2}_{\Tilde{\FS}_T}^{p-2}\big)\norm{ u_1- u_2}_{\Tilde{\FS}_T};
\end{equation*}

\medskip

\emph{Claim 2}.
\begin{equation*}
    \sup_{t\in[0,T]}t^{\frac{1}{4}}\norm{\grad\mathcal{T}( u_1)(t)-\grad\mathcal{T}( u_2)(t)}_{L^2}
    \leqslant C t^{\frac{3-p}{2}}\big(\norm{ u_1}_{\Tilde{\FS}_T}^{p-2}+\norm{ u_2}_{\Tilde{\FS}_T}^{p-2}\big)\norm{ u_1- u_2}_{\Tilde{\FS}_T}.
\end{equation*}

\medskip

\emph{Proof of Claim 1}: For any $0 \le t \le T$, 
\begin{eqnarray*}
&&\left\|\calT(u_1)-\calT(u_2) \right\|_{L^2}\leq \int_0^t \left\|\nabla e^{-(t-s) \FL} \left(|\nabla u_1|^{p-2} \nabla u_1-|\nabla u_2|^{p-2} \nabla u_2 \right) \right\|_{L^2} ds \\
&& \quad \quad \quad \quad  \leq C \int_0^t (t-s)^{-\frac{1}{4}} \cdot (t-s)^{-\frac{p-2}{4}} \left\||\nabla u_1|^{p-2} \nabla u_1-|\nabla u_2|^{p-2} \nabla u_2 \right\|_{L^{\frac{2}{p-1}}} ds \\
&&  \quad \quad \quad \quad  \leq C \int_0^t (t-s)^{-\frac{p-1}{4}} \left\| \left|\nabla u_1-\nabla u_2 \right| \left(|\nabla u_1|^{p-2}+|\nabla u_2|^{p-2} \right) \right\|_{L^{\frac{2}{p-1}}} ds \\
&& \quad \quad \quad \quad \leq C \int_0^t (t-s)^{-\frac{p-1}{4}}  \|\nabla u_1-\nabla u_2\|_{L^2} \left( \|\nabla u_1\|_{L^2}^{p-2}+\|\nabla u_2 \|_{L^2}^{p-2} \right)ds  \\
&& \quad \quad \quad \quad \leq C\int_0^t  (t-s)^{-\frac{p-1}{4}} s^{-\frac{p-1}{4}} \|u_1-u_2\|_{\Tilde{\FS}_T} \left(\|u_1\|_{\Tilde{\FS}_T}^{p-2}+\|u_2\|^{p-2}_{\Tilde{\FS}_T} \right) ds \\
&& \quad \quad \quad \quad \leq Ct^{\frac{3-p}{2}} \left(\|u_1\|_{\Tilde{\FS}_T}^{p-2}+\|u_2\|^{p-2}_{\Tilde{\FS}_T} \right)\|u_1-u_2\|_{\Tilde{\FS}_T},
\end{eqnarray*}
where in the above estimate, we have used the following elementary estimate: for any $a, b \in \R^2$ and $p \ge 2$, 
$$
\left| |a|^{p-2}a-|b|^{p-2} b\right| \leq C_p (|a|^{p-2}+|b|^{p-2})|a-b|.
$$

\medskip

\emph{Proof of Claim 2:}  The proof is similar to the one of \emph{Claim 2} in the previous lemma:
\begin{eqnarray*}
&&\left\|\nabla \calT(u_1)-\nabla T(u_2) \right\|_{L^2} \leq \int_0^t \left\|\lap e^{-(t-s)\FL} \left(|\nabla u_1|^{p-2} \nabla u_1-|\nabla u_2|^{p-2}\nabla u_2 \right)\right\|_{L^2} ds \\
&& \quad \quad \quad \quad \quad \quad \leq \int_0^t (t-s)^{-\frac{p}{4}} \|\nabla u_1-\nabla u_2\|_{L^2} \left(\|\nabla u_1\|_{L^2}^{p-2}+\|\nabla u_2\|_{L^2}^{p-2} \right) ds \\
&&  \quad \quad \quad \quad \quad \quad \leq Ct^{\frac{5-2p}{4}} \left(\|u_1\|_{\Tilde{\FS}_T}^{p-2}+\|u_2\|^{p-2}_{\Tilde{\FS}_T} \right)\|u_1-u_2\|_{\Tilde{\FS}_T}.
\end{eqnarray*}
Again, multiplying $t^{\frac{1}{4}}$ on the both sides of the above estimate yields 
\begin{equation*}
    t^{\frac{1}{4}}\norm{\grad\mathcal{T}(u_1)-\grad\mathcal{T}(u_2)}_{L^2}
    \leq C t^{\frac{3-p}{2}}\big(\norm{ u_1}_{\Tilde{\FS}_T}^{p-2}+\norm{ u_2}_{\Tilde{\FS}_T}^{p-2}\big)\norm{ u_1- u_2}_{\Tilde{\FS}_T},
\end{equation*}
for any $t\in[0,T]$. The proof of Lemma \ref{lem:L-cts of operator} is complete. 
\end{proof}

Finally, we show that the mild solution given in Theorem~\ref{thm:local mild soln} is also a weak solution in the sense of Definition~\ref{def:mild and weak soln}. 

\begin{proposition}
\label{prop:mild equi with weak}
Let $2<p<3$ and $u$ be the mild solution on $[0,T]$ given in Theorem~\ref{thm:local mild soln}. Then it is also a weak solution of~\eqref{e:4 order eqn} on $[0,T]$. Moreover, the following energy identity holds: for any $t\in[0,T)$, 
\begin{equation}
\label{e:energy identity}
    \norm{u(t)}_{L^2}^2 + 2\int_0^t\,\norm{\lap u(s)}_{L^2}^2\,ds
    =\norm{u_0}_{L^2}^2 + 2\int_0^t\,\norm{\grad u(s)}_{L^p}^p ds.
\end{equation}
\end{proposition}
\begin{proof}
First of all, observe that for any $0<\varepsilon \le t \le T$, one can rewrite the mild solution $u(t)$ with following expression on $[\epsilon, T]$:
\begin{equation}
\label{e:mild with epsilon}
    u(t)=e^{-(t-\epsilon)\FL}u(\epsilon)\,-\,\int_{\epsilon}^{t}\,\nabla e^{-(t-s)\FL}\, F(\grad u(s))\,ds,
\end{equation}
where we recall that $F(\xi)=|\xi|^{p-2}\xi, \xi \in \R^2$. Here the equality holds as functions in $C([\epsilon,T]; L^2(\T^2))$. Indeed, by the property of semigroup and Definition~\ref{def:mild and weak soln}, it follows that:
\begin{align*}
    &u(t)-u(\epsilon)=\left(e^{-(t-\epsilon)\FL}-I\right)\left(e^{-\epsilon\FL}u_0-\int_0^{\epsilon}\grad e^{-(\epsilon-s)\FL}F\left(\grad u\right)\left(s\right)ds\right)\\
    &  \quad \quad \quad  \quad \quad \quad \quad \quad \quad  -\int_{\epsilon}^{t}\grad e^{-(t-s)\FL}F\left(\grad u\right)\left(s\right)ds \\
    &\quad \quad \quad \quad \quad =\left(e^{-(t-\epsilon)\FL}-I\right)u(\epsilon)-\int_{\epsilon}^{t}\grad e^{-(t-s)\FL}F\left(\grad u\right)\left(s\right)ds,
\end{align*}
where $I$ is the identity map. Next, we show that this actually implies $u\in L^2([\epsilon,T]; H^2(\T^2))$. Since $u\in C([\epsilon,T]; L^2(\T^2))$, it is enough to check $\lap u\in L^2([\epsilon,T]\times\T^2)$, therefore we take $\lap$ to both sides of~\eqref{e:mild with epsilon}, and use the decay estimates of the semigroup $e^{-t\FL}$ also the fact that $u\in\Tilde{\FS}_T$ on $[\epsilon,T]$:
\begin{align*}
    \norm{\lap u(t)}_{L^2}
    &\leq C\Big((t-\epsilon)^{-\frac{1}{4}}\norm{u(\epsilon)}_{\dot{H}^1}+\int_{\epsilon}^{t}\,(t-s)^{-\frac{3}{4}} \cdot (t-s)^{-\frac{p-2}{4}} \cdot \norm{\grad u}_{L^{2}}^{p-1}\,ds\Big)\\
    &= C \left( (t-\varepsilon)^{-\frac{1}{4}} \|u(\varepsilon)\|_{\dot{H}^1}+\int_\varepsilon^t (t-s)^{-\frac{p+1}{4}} \|\nabla u\|_{L^2}^{p-1} ds\right) \\
    &\leq C\Big((t-\epsilon)^{-\frac{1}{4}}\epsilon^{-\frac{1}{4}}\norm{u}_{\Tilde{\FS}_T}+t^{\frac{2-p}{2}} \norm{u}_{\Tilde{\FS}_T}^{p-1}\Big). 
\end{align*}
Note for the last inequality, we implicitly use the fact that $p<3$. Also notice that the right hand side is square integrable in time for $p<3$, $\epsilon>0$. As some consequence, we have the following claims.

\medskip

\textit{Claim 1:} $F(\nabla u) \in L^2 \left( \left[\varepsilon, T \right]; L^{\frac{1}{p-2}}(\T^2) \right)$. 

\medskip

\textit{Claim 2:} $N(u) \in L^1 \left( \left[\varepsilon, T\right] \times \T^2 \right)$. 

\medskip

The first claim is a consequence of the Gagliardo-Nirenberg's inequality (note that $\frac{1}{2}< \frac{1}{p-1} < 1$): 
\begin{eqnarray} \label{20200702eq01}
&&\int_{\epsilon}^{T}\,\norm{F(\grad u)}^2_{L^{\frac{1}{p-2}}}\,ds= \int_{\epsilon}^{T}\,\norm{\grad u}^{2(p-1)}_{L^{\frac{p-1}{p-2}}}\,ds \nonumber \\
&& \quad \quad \leq C \int_\varepsilon^T \left( \left\|\lap u \right\|_{L^2}^2 \left\|u \right\|_{L^2}^{2(p-2)}+\|u\|_{L^2}^{2(p-1)} \right) ds  \nonumber \\
&& \quad \quad \leqslant C \left( \norm{u}_{C([\epsilon,T]; L^2(\T^2))}^{2(p-2)}\norm{u}_{L^2([\epsilon,T]; H^2(\T^2))}^2+T \norm{u}_{C([\epsilon,T]; L^2(\T^2))}^{2(p-1)} \right)
\end{eqnarray}
While for the second claim, we first compute that
$$
\left|N(u) \right|=\left| \nabla \cdot \left(|\nabla u|^{p-2} \nabla u \right) \right| \leq \left|\nabla u \right|^{p-2} \left|\lap u \right|+(p-2)\left|\nabla u \right|^{p-4}  \left| \left(\nabla u \right)^T \nabla^2 u \nabla u \right|, 
$$
where $v^T$ denotes the transpose of a vector $v \in \R^2$. This gives
\begin{eqnarray*}
&&\int_{\varepsilon}^T \left\|N(u) \right\|_{L^1}ds \leq \int_\varepsilon^T \left\|\left|\nabla u \right|^{p-2} \left|\lap u \right| \right\|_{L^1} ds \\
&&   \quad \quad  \quad \quad  \quad \quad  \quad \quad  \quad \quad  \quad \quad  \quad \quad  + (p-2)\int_\varepsilon^T \left\|\left|\nabla u \right|^{p-4}  \left| \left(\nabla u \right)^T \nabla^2 u \nabla u \right| \right\|_{L^1} ds \\
&& \quad \quad  \quad = \int_\varepsilon^T \left\|\left|\nabla u \right|^{p-2} \left|\lap u \right| \right\|_{L^1} ds+ (p-2)\int_\varepsilon^T \left\|\left|\nabla u \right|^{p-2} \cdot \frac{\left| \left(\nabla u \right)^T \nabla^2 u \nabla u \right|}{|\nabla u|^2} \right\|_{L^1} ds \\
&& \quad \quad  \quad  \leq  C \int_\varepsilon^T \norm{\grad u}_{L^{2(p-2)}}^{p-2}\norm{\lap u}_{L^2}\,ds   \\
&& \quad \quad \quad \leq C \int_\varepsilon^T \|\lap u \|_{L^2}^2 ds + C \int_\varepsilon^T \|\nabla u\|_{L^{2(p-2)}}^{2(p-2)} ds \\
&& \quad \quad \quad \le C\|u\|_{L^2([\varepsilon, T]; H^2(\T^2)}^2+C \|u\|_{\Tilde{S}_T}^{2(p-2)}<\infty. 
\end{eqnarray*}
Hence, the second claim is verified. 

\medskip

Let $\phi\in C_c^{\infty}([\epsilon,T)\times \T^2)$ and consider the $L^2$ pairing of $\phi$ with $u$. Observe that $e^{-t\FL}$ is a self-adjoint operator on $L^2$, and $\grad$ operator commutes with $e^{-t\FL}$. This gives 
\begin{align*}
    &\int_{\T^2}\,\phi(t)u(t)dx=\int_{\T^2}\,e^{-(t-\epsilon)\FL}\phi(t)u(\epsilon)dx\, \\
    &  \quad \quad \quad  \quad \quad \quad \quad \quad \quad  +\,\int_{\epsilon}^t\,\int_{\T^2}\,(e^{-(t-s)\FL}\grad\phi(t))\cdot F(\grad u(s))\,dx ds. 
\end{align*}
 Note that $\phi(t)$ and $\grad\phi(t)$ are in the domain of $\FL$ for all $t\in[\epsilon,T)$, and $\FL$ generates an analytic semigroup. Therefore, we have
\begin{equation*}
    \lim_{h\rightarrow 0}\frac{e^{-(t+h-\epsilon)\FL}\phi(t)-e^{-(t-\epsilon)\FL}\phi(t)}{h}=-e^{-(t-\epsilon)\FL}\FL\phi(t),
\end{equation*}
\begin{equation*}
    \lim_{h\rightarrow 0}\frac{e^{-(t-\epsilon)\FL}\phi(t+h)-e^{-(t-\epsilon)\FL}\phi(t)}{h}=e^{-(t-\epsilon)\FL}\partial_t \phi(t),
\end{equation*}
both strongly in $L^2(\T^2)$, similarly if $\phi$ is replaced by $\grad\phi$ . Since $e^{-t\FL}$ is a strongly continuous semigroup on $L^2(\T^2)$, it follows that $\frac{d}{dt}e^{-(t-\epsilon)\FL}\phi(t)\in C([\epsilon,T);L^2(\T^2))$ and by Leibniz rule:
\begin{equation*}
    \frac{d}{dt}e^{-(t-\epsilon)\FL}\phi(t)=
    e^{-(t-\epsilon)\FL}\partial_t\phi(t)-\FL e^{-(t-\epsilon)\FL}\phi(t).
\end{equation*}
In particular, the pairing of $e^{-(t-\epsilon)\FL}\phi(t)$ with any function $f(t)$ that is absolutely continuous in $t\in[\epsilon,T)$ with values in $L^2(\T^2)$, is differentiable a.e. in $t$, the derivative is integrable in time, and 
\begin{eqnarray}
\label{e:derivative of pairing}
&&\frac{d}{dt}\int_{\T^2}\,e^{-(t-\epsilon)\FL}\phi(t)f(t)\,dx =\int_{\T^2}\,\Big(-\FL e^{-(t-\epsilon)\FL}\phi(t)f(t)+e^{-(t-\epsilon)\FL}\partial_t\phi(t)f(t) \nonumber \\
&& \quad \quad\quad \quad\quad \quad\quad \quad\quad \quad\quad \quad\quad \quad \quad \quad\quad \quad\quad\quad \quad \quad \quad+e^{-(t-\epsilon)\FL}\phi(t)f'(t)\Big)\,dx.
\end{eqnarray}
Similarly with $\phi$ being replaced by $\nabla \phi$, 
\begin{eqnarray}
\label{e:derivative of pairing with grad}
&&\frac{d}{dt}\int_{\T^2}\,e^{-(t-\epsilon)\FL}\grad\phi(t)\cdot f(t)\,dx =\int_{\T^2}\,\Big(-\FL e^{-(t-\epsilon)\FL}\grad\phi(t)\cdot f(t) \nonumber \\
&&  \quad \quad \quad \quad \quad \quad \quad \quad \quad \quad \quad  +e^{-(t-\epsilon)\FL}\partial_t\grad\phi(t)\cdot f(t) +e^{-(t-\epsilon)\FL}\grad\phi(t)\cdot f'(t)\Big)\,dx.
\end{eqnarray}
Recall the following property from functional analysis: if $X$ is a separable and reflexive Banach space, then for a $X$-valued function of time is absolute continuous if and only if it has weak derivative for almost every time, and is integrable in the B\"{o}chner sense. By using~\eqref{e:derivative of pairing} and~\eqref{e:derivative of pairing with grad},  we have
\begin{align*}
    &\frac{d}{dt}\int_{\T^2}\,\phi(t)u(t)\,dx=\int_{\T^2}\,\Big(e^{-(t-\epsilon)\FL}\partial_t\phi(t)-e^{-(t-\epsilon)\FL}\FL\phi(t)\Big)u(\epsilon)\,dx\\\notag
    & \quad \quad  +\int_{\T^2}\,\int_{\epsilon}^{t}\,\Big(e^{-(t-s)\FL}\partial_t\big(\grad\phi(t)\big)-\FL e^{-(t-s)\FL}\grad\phi(t)\Big)\cdot\Big(\abs{\grad u(s)}^{p-2}\grad u(s)\Big)\,dt dx\notag\\\notag
    & \quad \quad  +\int_{\T^2}\,\grad\phi(t)\cdot\Big(\abs{\grad u(t)}^{p-2}\grad u(t)\Big)\,dx.
\end{align*}
Rearranging the terms, we see that 
\begin{align*}
    &\frac{d}{dt}\int_{\T^2}\,\phi(t)u(t)\,dx\\
    &\quad =\int_{\T^2}\,\partial_t\phi(t)\Big(e^{-(t-\epsilon)\FL}u(\epsilon)-\int_{\epsilon}^t\,\nabla  e^{-(t-s)\FL} \big(\abs{\grad u}^{p-2}\grad u\big)(s)\,ds\Big)\,dx\notag\\
    &\quad \quad \quad  -\int_{\T^2}\,\FL\phi(t)\Big(e^{-(t-\epsilon)\FL}u(\epsilon)-\int_{\epsilon}^t\,\nabla e^{-(t-s)\FL} \big(\abs{\grad u}^{p-2}\grad u\big)(s)\,ds\Big)\,dx\notag\\
    &\quad \quad \quad  +\int_{\T^2}\,\grad\phi(t)\cdot\Big(\abs{\grad u(t)}^{p-2}\grad u(t)\Big)\,dx\notag\\\notag
    &\quad  =\int_{\T^2}\,\big(\partial_t\phi(t)-\FL\phi(t)\big)u(t)dx\,+\,\int_{\T^2}\,\grad\phi(t)\cdot\Big(\abs{\grad u(t)}^{p-2}\grad u(t)\Big)\,dx.
\end{align*}
Integrating the above identity from $\epsilon$ to $t$ for $t\in(\epsilon,T)$ yields: 
\begin{eqnarray}
\label{e: prelimit identity}
&&\int_{\T^2}\,\phi(t)u(t)\,dx\,-\,\int_{\T^2}\,\phi(\epsilon)u(\epsilon)\,dx =\int_{\epsilon}^t\int_{\T^2}\,\partial_{t}\phi(s)u(s)\,dx ds\, \nonumber \\
&& \quad \quad \quad \quad -\,\int_{\epsilon}^t\int_{\T^2}\,\lap\phi(s)\lap u(s)dx ds  +\,\int_{\epsilon}^t\int_{\T^2}\,\grad\phi(s)\cdot\Big(\abs{\grad u(s)}^{p-2}\grad u(s)\Big)\,dx ds. 
\end{eqnarray}
We claim that $\partial_t u \in L^2 \left([\varepsilon, T); H^{-2} \right)$. To see this, we note that it suffices to show that \eqref{e: prelimit identity} is well-defined for $\phi \in L^2 \left( \left[ \varepsilon, t \right]; H^2 \right)$ with $\partial_t \phi \in L^2 \left( \left[ \varepsilon, t \right]; H^{-2} \right)$. The key observation here is to show that last term in \eqref{e: prelimit identity} is well-defined under such an assumption. Indeed, 
\begin{eqnarray*}
&& \,\int_{\epsilon}^t\int_{\T^2}\,\grad\phi(s)\cdot\Big(\abs{\grad u(s)}^{p-2}\grad u(s)\Big)\,dx ds=\int_\varepsilon^t \int_{\T^2} \nabla \phi(s) \cdot F \left(\nabla u(s) \right) dxds \\
&& \quad \quad \quad \leq \int_\varepsilon^t \left\|\nabla \phi(s)\right\|_{L^{\frac{1}{3-p}}} \cdot \left\|F(\nabla u(s)) \right\|_{L^{\frac{1}{p-2}}} ds \\
&& \quad \quad \quad \leq \left(\int_\varepsilon^t \left\|\nabla \phi\right\|_{L^{\frac{1}{3-p}}}^2 ds \right)^{\frac{1}{2}} \cdot \left(\int_\varepsilon^t \left\|F(\nabla u(s) ) \right\|_{L^{\frac{1}{p-2}}}^2 ds \right)^{\frac{1}{2}} \\
&& \quad \quad \quad \leq C\left(\int_\varepsilon^t \left\|\lap \phi\right\|_{L^2}^2 +\|\phi\|_{L^2}^2ds \right)^{\frac{1}{2}} \cdot  \left\| F(\nabla u) \right\|_{L^2 \left( \left[ \varepsilon, T \right]; L^{\frac{1}{p-2}}(\T^2) \right)} \\
&& \quad \quad \quad \leq C\|\phi\|_{L^2 \left(\left[\varepsilon, t \right]; H^2 \right)} \left\| F(\nabla u) \right\|_{L^2 \left( \left[ \varepsilon, T \right]; L^{\frac{1}{p-2}}(\T^2) \right)}<\infty, 
\end{eqnarray*} 
where we use \eqref{20200702eq01} (see, \emph{Claim 1}) and Gagliardo-Nirenberg’s inequality in the second last inequality. 

\medskip

Since $u\in L^2([\epsilon,T);H^2)$ with $\partial_t u \in L^2 \left([\varepsilon, T); H^{-2} \right)$, there exists a sequence $\{u_{m}\}\subset C^1([\epsilon,T);H^2)$ such that $u_{m}\rightarrow u$ strongly in $C((\epsilon,T);L^2)\cap L^2((\epsilon,T);H^2)$ and $\partial_t u_{m}\rightarrow\partial_t u$ weakly in $L^2((\epsilon,T);H^{-2})$ as $m \rightarrow\infty$. This in particular gives
\begin{eqnarray*}
    && \Big\vert\int_{\epsilon}^t\int_{\T^2}\,(\grad u_m-\grad u)F(u(s))\,dx ds\Big\vert \\
    && \quad \quad \quad \quad \leq\norm{\grad u_{m}-\grad u}_{L^2\left((\epsilon,T),L^{\frac{1}{3-p}}\right)}\norm{F(\grad u)}_{L^2\left([\epsilon,T),L^{\frac{1}{p-2}}\right)}\\
    && \quad \quad \quad \quad  \leqslant C\norm{u_{m}-u}_{L^2\left([\epsilon, T);H^2\right)}\norm{F(\grad u)}_{L^2\left([\epsilon,T); L^{\frac{1}{p-2}}\right)},
\end{eqnarray*}
which converges to $0$ as $m \to \infty$. Therefore, by plugging $u_m$ as test functions in~\eqref{e: prelimit identity}, and letting $m\rightarrow\infty$, we derive that
\begin{align} \label{20210702eq02} 
    \norm{u(t)}_{L^2}^2-\norm{u(\epsilon)}_{L^2}^2=\int_{\epsilon}^t\int_{\T^2}\,u(s)\partial_t u(s)dxds\,-\,\int_{\epsilon}^{t}\,\norm{\lap u(s)}_{L^2}^2\,ds\,+\,\int_{\epsilon}^t\norm{\grad u}_{L^p}^p\,ds.
\end{align}
Since $u\in L^2([\epsilon,T);H^2)$ with $\partial_t u\in L^2([\epsilon,T);H^{-2})$, we have
$t\rightarrow\norm{u(t)}_{L^2}^2$ is absolutely continuous and an application of the Fundamental Theorem of Calculus gives: 
\begin{equation*}
    \int_{\epsilon}^t\int_{\T^2}\,u(s)\partial_t u(s)\,dx ds =\frac{1}{2}\int_{\epsilon}^t\int_{\T^2}\,\partial_t (u(s))^2\,dx ds=\frac{1}{2}\left(\norm{u(t)}_{L^2}^2-\norm{u(\epsilon)}_{L^2}^2\right), 
\end{equation*}
which together with \eqref{20210702eq02} implies 
\begin{equation}
\label{e:energy identity with e}
    \norm{u(t)}_{L^2}^2\,-\,\norm{u(\epsilon)}_{L^2}^2\,+\,2\int_{\epsilon}^t\,\norm{\lap u(s)}_{L^2}^2\,ds = 2\int_{\epsilon}^t\,\norm{\grad u}_{L^p}^p\,ds.
\end{equation}

Next, we show that we can take $\epsilon\rightarrow 0$ in the above identity and obtain an energy inequality valid on $[0,T)$. Indeed, the Gagliardo-Nirenberg inequality yields 
\begin{equation*}
    \int_{\epsilon}^t\,\norm{\grad u}_{L^p}^p\,ds
    \leq \int_{\epsilon}^t\,\norm{\lap u(s)}_{L^2}^2\,ds\,+\, C\int_{\epsilon}^t\ \left(\norm{u(s)}_{L^2}^{\frac{2}{3-p}}+\|u(s)\|_{L^2}^p \right)ds. 
\end{equation*}
Plugging the above estimate into~\eqref{e:energy identity with e}, we see that
\begin{eqnarray*}
    \int_{\epsilon}^t\,\norm{\lap u(s)}_{L^2}^2\,ds%
    &\leq& \norm{u(\epsilon)}_{L^2}^2\,-\,\norm{u(t)}_{L^2}^2\,+\,C\int_{\epsilon}^t\left(\norm{u(s)}_{L^2}^{\frac{2}{3-p}}+\|u(s)\|_{L^2}^p \right)ds \\
    &\leq& C \left(\|u\|_{\Tilde{\mathcal S}_T}+T\|u\|_{\Tilde{\mathcal S}_T}^{\frac{2}{3-p}}+T\|u\|_{\Tilde{\mathcal S}_T}^p \right), 
\end{eqnarray*}
where the right hand side in the above estimate above is uniform in $\varepsilon$. Note that this also implies $u\in L^2([0,T);H^2)$. 

Therefore, by the Dominated convergence theorem, we are able to let $\epsilon \to 0$ in~\eqref{e:energy identity with e} to recover~\eqref{e:energy identity}. Here, in particular we have used the continuity of the function $\|u(t)\|_{L^2}$ at $t=0$ (one can easily verify this by the definition of mild solutions). Similarly, since $u\in L^2([0,T);H^2)$, the right hand side of~\eqref{e: prelimit identity} is uniformly bounded with respect to $\epsilon$ with $\phi\in C_c^\infty((0,T);H^2)$. Thus, $(\partial_t u)\chi_{[\epsilon,T)}$ converges weakly to $\partial_t u$ in $L^2((0,T);H^{-2})$. Finally, by taking $\epsilon\rightarrow 0$ in~\eqref{e: prelimit identity} with a test function $\phi\in C_c^{\infty}([0,T]\times\T^2)$ (by Dominated convergence theorem again), we get the weak formulation in~\eqref{e:weak soln}. Hence the mild solution $u$ is also a weak solution on $[0,T]$. 
\end{proof}


\section{Epitaxial thin film growth equation with advection} \label{fengyuxiaogege}
In this section, we study the thin film equation with advection on $\T^2$. Accordingly, we add an advection term to the equation~\eqref{e:4 order eqn}:
\begin{equation}
\label{e:adv 4 order eqn}
    \partial_t u + v\cdot\grad u = -\lap^2 u -\grad\cdot(\abs{\grad u}^{p-2}\grad u),
    \qquad u(0,x)= u_0(x)
\end{equation}
where $v$ is a given divergence free vector field. The plan of this section is as follow. To begin with, we establish the local existence of solution to equation~\eqref{e:adv 4 order eqn} in Section~\ref{sec:local exists for adv}. Next, in Section~\ref{sec:global existence} we show that if the velocity field $v$ has small \emph{dissipation time} (see Definition~\ref{def:dissipationTime}), then the corresponding solution to~\eqref{e:adv 4 order eqn} can be extended to be a global one. Further, the $L^2$ norm of the global solution will decay exponentially fast on the whole time interval.
\subsection{Local existence with $L^2$ initial data}
\label{sec:local exists for adv}
As in the non-advective case, equation~\eqref{e:adv 4 order eqn} admits an unique local mild solution in $\Tilde{\FS}_T$ if $v\in L^{\infty}([0,\infty);L^2(\T^2))$ is provided. Moreover, one can also check that the mild solution is equivalent to a weak solution and satisfies certain energy identity. The proofs are very similar to the non-advective case. Therefore, we only sketch the proof to these results and leave the details to the interested reader.
\begin{definition}
\label{def:mild and weak soln with advection}
\begin{enumerate}
\item [(1).] For $p>2$ and $v\in L^{\infty}([0,\infty);L^2(\T^2))$, a function $ u\in C([0,T];L^{2}(\mathbb{T}^2))$, $T>0$, such that $\grad u$ is locally integrable, is called a \emph{mild solution} of~\eqref{e:adv 4 order eqn} on $[0,T]$ with initial data $ u_0\in L^{2}(\T^2)$, if for any $0\leq t\leq T$,
\begin{equation}
\label{e:adv mild soln}
     u(t)=\mathcal{T}_{v}(u)(t)\defeq e^{-t\FL} u_0\,-\,\int_0^t\,\grad e^{-(t-s)\FL}(\abs{\grad u}^{p-2}\grad u)\,ds\,-\, \int_0^t\,e^{-(t-s)\FL}(v\cdot\grad u)\,ds,
\end{equation}
holds pointwisely in time with values in $L^2$, where the integral is defined in the B\"{o}chner sense.

\medskip

\item [(2). ] For $p>2$ and $v\in L^{\infty}([0,\infty);L^2(\T^2))$, a function $ u\in L^{\infty}([0,T];L^{2}(\T^{2}))\cap L^{2}([0,T];H^{2}(\T^{2}))$ is called a \emph{weak solution} of~\eqref{e:adv 4 order eqn} on $[0,T)$ with initial value $ u(0)= u_0\in L^{2}(\T^{2})$ if for all $\phi\in C_c^{\infty}([0,T)\times\T^{2})$,
\begin{align}
&\int_{T^2}\,u_0\phi(0)\,dx\, +\, \int_0^T\,\int_{\T^2}u\,\partial_t \phi\,dx dt \\\notag
&=\int_0^T\,\int_{\T^2}\,\lap u\lap\phi dx dt\, -\, \int_0^T \int_{\T^2}\,\abs{\grad u}^{p-2}\grad u\cdot\grad\phi\,dx dt \,+\,\int_0^T \int_{\T^2}\,\phi v\cdot\grad u\,dx dt ,
\end{align}
and $\partial_t u\in L^{2}([0,T]; H^{-2}(\T^2))$. 
\end{enumerate}
\end{definition}

\begin{theorem}
\label{thm:local mild soln with advection}
Let $u_0\in L^{2}(\T^2)$, $2<p< 3$, and $v\in L^{\infty}([0,\infty);L^2(\T^2))$. Then there exists $0<T\leqslant 1$, which only depends on $\norm{u_0}_{L^2}$ and $\sup\limits_{t\in[0,\infty)}\norm{v}_{L^2}$ such that 
equation~\eqref{e:adv 4 order eqn} admits a mild solution $u$ on $[0,T]$, which is unique in $\Tilde{\FS}_T$ .
\end{theorem}
\begin{proof}
Observe that if $v\in L^{\infty}([0,\infty);L^2(\T^2))$, then for $u\in\Tilde{\FS}_T$, we have the following estimates for the advection term:
\begin{align*}
    \Big\Vert\int_0^t\,e^{-\left(t-s\right)\FL}(v\cdot\grad u)\,ds\Big\Vert_{L^2}
    &\leqslant\int_{0}^{t}(t-s)^{-\tfrac{1}{4}}\norm{v\cdot\grad u}_{L^1}ds\\
    &\leqslant\norm{v}_{L^\infty([0,\infty);L^2)}\int_{0}^{t}\left(t-s\right)^{-\tfrac{1}{4}}\cdot s^{-\tfrac{1}{4}}\left(s^{\tfrac{1}{4}}\norm{\grad u\left(s\right)}_{L^2}\right)\,ds\\
    &\leqslant t^{\tfrac{1}{2}} \cdot \norm{v}_{L^\infty\left([0,\infty);L^2\right)}\norm{ u}_{\Tilde{\FS}_T},
\end{align*}
where we used Lemma~\ref{lem: semi group on divergence} with $p=3$ in the first inequality above. Similarly,
\begin{equation*}
    t^{\tfrac{1}{4}} \cdot \Big\Vert\int_0^t\,\grad e^{-\left(t-s\right)\FL}\left(v\cdot\grad u\right)\,ds\Big\Vert_{L^2}\leqslant t^{\tfrac{1}{2}} \cdot \norm{v}_{L^\infty\left([0,\infty);L^2\right)} \cdot \norm{ u}_{\Tilde{\FS}_T}.
\end{equation*}
Moreover, for $u_1,u_2\in\Tilde{\FS}_T$ we have
\begin{align*}
   \Big\Vert\int_0^t\,e^{-(t-s)\FL}\left(v\cdot\grad \left(u_1-u_2\right)\right)\,ds\Big\Vert_{L^2}\leqslant t^{\tfrac{1}{2}}\cdot \norm{v}_{L^\infty([0,\infty);L^2)} \cdot \norm{u_1-u_2}_{\Tilde{\FS}_T},\\
   t^{\tfrac{1}{4}} \cdot\Big\Vert\int_0^t\,\grad e^{-(t-s)\FL}\left(v\cdot\grad \left(u_1-u_2\right)\right)\,ds\Big\Vert_{L^2}\leqslant t^{\tfrac{1}{2}} \cdot \norm{v}_{L^\infty([0,\infty);L^2)} \cdot \norm{u_1-u_2}_{\Tilde{\FS}_T}.
\end{align*}
By using the above estimates to modify the proof of Theorem~\ref{thm:local mild soln}, one can check that for $0< T\leq 1$, and $u,u_1,u_2\in\Tilde{\FS}_T$ there exists constant $C_0$ such that 
\begin{equation*}
    \norm{\mathcal{T}_{v}( u)}_{\Tilde{\FS}_T}\leq C_0\Big(\norm{ u_0}_{L^2}+
    T^{\tfrac{3-p}{2}}\norm{ u}_{\Tilde{\FS}_T}^{p-1}+T^{\tfrac{1}{2}}\norm{v}_{L^{\infty}([0,T];L^2)}\norm{ u}_{\Tilde{\FS}_T}\Big),
\end{equation*}
and 
\begin{align*}
    &\norm{\mathcal{T}_{v}(u_1)-\mathcal{T}_{v}(u_2)}_{\Tilde{\FS}_T}\\
    &\leq C_0 T^{\tfrac{3-p}{2}}\left(\norm{ u_1}_{\Tilde{\FS}_T}^{p-2}+\norm{ u_2}_{\Tilde{\FS}_T}^{p-2}+T^{\tfrac{p-2}{2}}\norm{v}_{L^{\infty}([0,\infty);L^2)}\right)\norm{ u_1- u_2}_{\Tilde{\FS}_T}.
\end{align*}
Therefore, for $R\geq 2C_0\norm{u_0}_{L^2}$, if
\begin{equation}
\label{e:expression of T}
    T\leqslant \min\Big\{1, (4 C_0 R^{p-2}+1)^{\frac{-2}{3-p}},\big(16C_0^2\norm{v}_{L^{\infty}([0,\infty);L^2)}^2+1\big)^{-1}\Big\},
\end{equation}
then
\begin{subequations}
\begin{align*}
    \norm{\mathcal{T}_{v}(u)}_{\Tilde{\FS}_T}&\leq R,\qquad\forall u\in\mathcal{B}_{R}(0);\\
    \norm{\mathcal{T}_{v}(u_1)-\mathcal{T}(u_2)}_{\Tilde{\FS}_T}
    &\leq q \norm{u_1-u_2}_{\Tilde{\FS}_T},\qquad\text{with}\quad q<1.
\end{align*}
\end{subequations}
By the Banach contraction mapping theorem, there is a unique fixed point of $\mathcal{T}_{v}$ in $\mathcal{B}_{R}(0)$, which is the unique mild solution to~\eqref{e:adv 4 order eqn} with initial data $u_0$.
\end{proof}
\begin{corollary}
\label{cor:L2 criterion}
With the same assumptions in Theorem~\ref{thm:local mild soln with advection}, if $T^*$ is the maximal time of existence of the mild solution to~\eqref{e:adv 4 order eqn}, then 
\begin{equation*}
    \limsup_{t\rightarrow T_-^*}\norm{u(t)}_{L^2(\T^2)}=\infty.
\end{equation*}
Otherwise, $T^*=\infty$.
\end{corollary}

\begin{proposition}
\label{prop:adv mild equi with weak}
Let $u$ be the mild solution on $[0,T]$ given in Theorem~\ref{thm:local mild soln with advection}. Then $u$ is also a weak solution of~\eqref{e:adv 4 order eqn} on $[0,T]$, and satisfies the energy identity for any $t\in[0,T)$:
\begin{equation}
\label{e:energy identity for adv}
    \norm{u(t)}_{L^2}^2 + 2\int_0^t\,\norm{\lap u(s)}_{L^2}^2\,ds
    =\norm{u_0}_{L^2}^2 + 2\int_0^t\,\norm{\grad u(s)}_{L^p}^p\, ds.
\end{equation}
\end{proposition}
The proof of Corollary~\ref{cor:L2 criterion} and Proposition~\ref{prop:adv mild equi with weak} are similar to the proof of non-advective case, and hence we omit it. 
\begin{remark}
\label{rmk:redefine N}
\begin{enumerate}
    \item [(1).] By \eqref{e:expression of T}, we can conclude that the length of $T$ is non-decreasing as $\norm{u_0}_{L^2}$ decreases.
    
    \medskip
    
    \item [(2).]In the proof of Proposition~\ref{prop:mild equi with weak}, we showed that when $2<p<3$, $N(u)$ actually belongs to $L^1([0,T]\times \T^2)$, and this conclusion remains hold for the advective case. Therefore, in this regime we are allowed to rewrite the mild solution~\eqref{e:adv mild soln} into the following way:
    \begin{eqnarray*}
     u(t) %
     &=& e^{-t\FL} u_0\,-\,\int_0^t\,e^{-(t-s)\FL}\grad\cdot\left(\abs{\grad u}^{p-2}\grad u\right)\,ds \\
     && \quad \quad \quad \quad \quad -\, \int_0^t\,e^{-(t-s)\FL}\left(v\cdot\grad u\right)\,ds.
\end{eqnarray*}
Note that this alternative form plays a crucial role in the sequel (see, e.g.,  Lemma \ref{lem: mean small H2}).
\end{enumerate} 
\end{remark}

\subsection{Global existence with advecting flows with small dissipation time}
\label{sec:global existence}
In the second part of this section, our goal is to show the existence of global mild solution in $\Tilde{\FS}_T$ space to equation~\eqref{e:adv 4 order eqn} when $v\in L^{\infty}\left([0,\infty), W^{1,\infty}(\T^d)\right)$ and the \emph{dissipation time} of it satisfies certain constraint (see condition \eqref{e:flow condition}). Let $\bar{u}$ denotes the mean of the solution, then observe that equation~\eqref{e:adv 4 order eqn} is mean conserved, and $u-\bar{u}$ satisfies  equation~\eqref{e:adv 4 order eqn} as well.  Therefore, without loss of generality, we may assume $u_0\in L_0^2$, where $L_0^2$ is the space of all mean-zero, square integrable functions on the torus. 

Our main result of this section is in the same spirit and can be regarded as certain counterparts to those in~\cite{iyer2019convection}, in which the authors proved exponential decay of solutions to a large class of nonlinear second order parabolic equations. However, the assumptions required to apply the result in~\cite{iyer2019convection} do not apply to the nonlinear term $N(u)$ in~\eqref{e:adv 4 order eqn}, more precisely, $N(u)$ belongs to $L^1(\T^d)$ instead of $L^2(\T^d)$, and it cannot be bounded in terms of $\norm{u}_{L^2}$ and $\norm{\lap u}_{L^2}^2$. To overcome this difficulty, we estimate the difference between the solution of \eqref{e:adv 4 order eqn} and \eqref{e:4 order eqn} with $L^1$ initial data for short time intervals, and show that both the difference and the solution of \eqref{e:4 order eqn} can be controlled by $\norm{u}_{L^2}$ and $\norm{\lap u}_{L^2}^2$.
Finally, combining all these ingredients, we show that the $L^2$ norm of the solution to~\eqref{e:adv 4 order eqn} will decay exponentially fast, if the \emph{dissipation time} of $v$ satisfies specific constraint.

Now we turn to the detail. Let $\mathcal S_{s, s+t}$ be the solution operator for advective hyper diffusion equation, which is the linear part of~\eqref{e:adv 4 order eqn}:
\begin{equation}
\label{e:adv hyper diff}
    \partial_t\phi + v\cdot\grad\phi + \Delta^2\phi = 0.
\end{equation} 
That is, the function $\phi(s+t)=S_{s, s+t}\phi(s)$ solves~\eqref{e:adv hyper diff} with initial data $\phi(s)$ and periodic boundary condition.

Now we are ready to introduce the term \emph{dissipation time}. Heuristically, it is the smallest time that the system halves the energy ($L^2$ norm) of the solution to~\eqref{e:adv hyper diff}.

\begin{definition}
\label{def:dissipationTime}
The \emph{dissipation time} associated to the solution operator $\mathcal S_{s, s+t}$ is defined by
  \begin{equation}
  \label{e:dtimeDef}
    \tau^*(v) \defeq  \inf \set[\Big]{ t \geq 0 \st  \norm{ \mathcal S_{s, s+t} }_{L_0^2 \to L_0^2 } \leq \tfrac{1}{2} \text{ for all $s\geq 0$} }\,.
  \end{equation}
\end{definition}

We show that the mild solution to~\eqref{e:adv 4 order eqn} in space $\Tilde{\FS}_T$ can be extended over any finite time interval $[0,T]$ if we choose $v$ carefully enough with respect to the initial data. Indeed, Theorem~\ref{thm:local mild soln with advection} indicates that as long as the $L^2$ norm remains finite, the solution can always be extended to a longer time interval. Therefore, it suffices for us to show when the \emph{dissipation time} of $v$ is small enough, then the $L^2$ norm of the solution will not increase over the interval $[0,T]$.

We are ready to state the main theorem of this section.
\begin{theorem}
\label{thm:global existence}
For $2<p<3$, $v\in L^{\infty}\left([0,\infty), W^{1,\infty}(\T^d)\right)$, and $\mu>0$, let $u$ be the mild solution of \eqref{e:adv 4 order eqn} with initial data $u(0)=u_0\in L_0^2$. Then there exists a threshold value 
\begin{equation*}
T_1=T_1(\norm{u_0}_{L^2}, \mu, p)  
\end{equation*}
such that if 
\begin{equation}\label{e:flow condition}
\left(\norm{v}_{L^\infty([0,\infty);L^{\infty})}\left(\tau^*(v)\right)^{\frac{5}{4}}+\left(\tau^*(v)\right)^{\frac{3}{4}}\right)\leqslant T_1(\norm{u_0}_{L^2}, \mu, p),
\end{equation}
then there exist a constant $\beta>0$, such that for any $t>0$ it holds
\begin{equation*}
    \norm{u(t)}_{L^2}\leq\beta e^{-\mu t}\norm{u_0}_{L^2}.
\end{equation*}
\end{theorem}
\begin{remark}
The dependence of $T_1$ can be computed explicitly, as can be seen from equation \eqref{e:dependence of T1}. It indicates that $T_1$ is a decreasing function of $\norm{u_0}_{L^2}$, and it also implies $\tau^*(v)\leqslant\min\left\{\frac{1}{10\mu},T_0^2(\norm{u_0}_{L^2})\right\}$.
\end{remark}
Before proving the theorem, we discuss the existence of flow that satisfies the assumption
\eqref{e:flow condition}.
\begin{proposition}\label{prop:existence of flow}
Let $v\in L^{\infty}([0,\infty);C^2(\T^d))$, and define $v_{A}=Av(x,At)$. If $v$ is weakly mixing flow with rate function $h$, then
\begin{equation*}
    \tau^{*}(v_A) \xrightarrow{A \to \infty} 0. 
\end{equation*}
If further $v$ is strongly mixing with rate function $h$, and 
\begin{equation*}
    t^4 h(t) \xrightarrow{A \to \infty} 0,
\end{equation*}
then
\begin{equation*}
    \norm{v_A}_{L^\infty([0,\infty);L^{\infty})}^{\frac{4}{5}}\tau^*(v_A) \xrightarrow{A \to \infty} 0.
\end{equation*}
\end{proposition}
\begin{remark}
\begin{itemize}
    \item [(1).] For the definition of weakly mixing flow and strongly mixing flow we refer reader to Definition 3.1 of~\cite{feng2020global}.
    \item [(2).] The proof of Proposition~\ref{prop:existence of flow} is in the same spirit of the proof of Proposition 1.4 of~\cite{feng2020global}, therefore we omit it here.
    \item[(3).]  There are other families of flows that we can manipulate and gain arbitrary small dissipation time, one typical type of flows are so-called mixing flows, which are the flows that can produce small scale structures. To the best of our knowledge, the existence of smooth time-independent or even time-periodic, mixing flows remains an open question (see, for example \cite{ElgindiZlatos19}). Nevertheless, there are many examples of (spatially) smooth, time-dependent flows can be used or be expected to provide small dissipation times, for example the alternating horizontal/vertical sinusodial shear flows with randomized phases. One can check \cite{feng2020global,Pierrehumbert94,naraigh2007bubbles,naraigh2008bounds} for more discussion.
\end{itemize}
\end{remark}
Now we turn to the proofs, we divide the proof of the main theorem into several lemmas. In the following lemma, we first show that the $L^2$ norm of $u$ grows continuously.

\begin{lemma}
\label{lem: L2 unif bdd}
When $2<p<3$, there is a constant $A_p>0$. For any $B>0$ define
\begin{equation}
\label{e:choice of T_0}
    T_0^2(B)=\int_{B}^{2B}\,\frac{y}{
    A_p y^{2/(3-p)}}\,dy.
\end{equation}
Let $u$ be a mild solution of~\eqref{e:adv 4 order eqn} on $[0,T]$, and $t_0\in[0,T]$. Then for any $t\in[t_0,t_0+T_0^2(\norm{u(t_0)}_{L^2})] \cap [0,T]$, the following estimate holds:
\begin{equation}
    \norm{u(t)}_{L^2}\leq 2\norm{u(t_0)}_{L^2}.
\end{equation}
\end{lemma}
\begin{remark}\label{rmk:decrease_of_T0}
Note that for $B>0$ and $2<p<3$, $T_0^2(B)$ is a decreasing function. Since we can check that
\begin{equation}
\frac{d}{dB}(T_0^2(B))=\frac{1}{A_p} \left(2^{\frac{2(2-p)}{3-p}}-1 \right)B^{\tfrac{1-p}{3-p}}<0.
\end{equation}
\end{remark}
 In the next two lemmas, we show that if the time average of $\norm{\lap u}_{L^2}^2$ is large, then $\norm{u}_{L^2}$ will decrease exponentially. Conversely, if the time average of $\norm{\lap u}_{L^2}^2$ is small, then the dissipation enhancement phenomenon will still cut down $\norm{u}_{L^2}$ exponentially with a comparable rate. We start with the first case in the following lemma.
\begin{lemma}
\label{lem:mean large H2}
When $2<p<3$, let $\mu>0$ and $u$ be a mild solution of \eqref{e:adv 4 order eqn} on $[0,T]$ with $t_0\in[0,T]$. If for some $\tau\in\left(0, T_0^2(\norm{u(t_0)}_{L^2})\right)$, with $[t_0, t_0+\tau]\subset[0,T]$, and further
\begin{equation}
\label{e:mean large of H2}
    \frac{1}{\tau}\int_{t_0}^{t_0+\tau}\,\norm{\lap u(t)}_{L^2}^2\,dt\geq 2A_p\left(2\norm{u(t_0)}_{L^2}\right)^{\tfrac{2}{3-p}}+2\mu\norm{u(t_0)}_{L^2}^2,
\end{equation}
where $A_p$ is the constant in Lemma \ref{lem: L2 unif bdd}, then
\begin{equation}
\label{e:exp decay of L2}
    \norm{u(t_0+\tau)}_{L^2}\leq e^{-\mu\tau}\norm{u(t_0)}_{L^2}.
\end{equation}
\end{lemma}
For the second case, if the time average of $\norm{\lap u}_{L^2}^2$ is small, we show that if $\tau^*(v)$ is small enough, then $\norm{u}_{L^2}^2$ will still decrease with a comparable rate.
\begin{lemma}
\label{lem: mean small H2}
When $2<p<3$, let $v\in L^\infty([0,\infty);L^{\infty})$, $\mu>0$, and $u$ be a mild solution of~\eqref{e:adv 4 order eqn} on $[0,T]$. Assume $t_0\in[0,T]$, then there exists a threshold value
\begin{equation}\label{e:dependence of T1}
  T_1=\min\left\{ \frac{2}{5C\left(A_p B^{\frac{2(p-2)}{3-p}}+\mu\right)^{\frac{p}{4}}B^{p-2}},\left(\frac{1}{10\mu}\right)^{\frac{3}{4}}, \left(T_0^2(\norm{u(t_0)}_{L^2})\right)^{\frac{3}{4}}\right\}
\end{equation}
such that if 
\begin{equation}
\left(\norm{v}_{L^\infty([0,\infty);L^{\infty})}\left(\tau^*(v)\right)^{\frac{5}{4}}+\left(\tau^*(v)\right)^{\frac{3}{4}}\right)\leqslant T_1(\norm{u_0}_{L^2}, \mu, p),
\end{equation}
and $t_0+\tau^*\leqslant T$ such that
\begin{equation}
\label{e:mean small of H2}
    \frac{1}{\tau^*}\int_{t_0}^{t_0+\tau^*}\,\norm{\lap u(t)}_{L^2}^2\,dt\leq 2A_p\left(2\norm{u(t_0)}_{L^2}\right)^{\tfrac{2}{3-p}}+2\mu\norm{u(t_0)}_{L^2}^2,
\end{equation}
where $A_p$ is the constant in Lemma \ref{lem: L2 unif bdd}. Then~\eqref{e:exp decay of L2} still holds at time $\tau=\tau^*$.
\end{lemma}
Temporally assuming Lemma~\ref{lem: L2 unif bdd}, Lemma~\ref{lem:mean large H2} and Lemma~\ref{lem: mean small H2}, we prove Theorem~\ref{thm:global existence} first.
\begin{proof}[Proof of Theorem \ref{thm:global existence}]
By Lemma~\ref{lem:mean large H2} and Lemma~\ref{lem: mean small H2}, estimate~\eqref{e:exp decay of L2} always hold for $\tau=2\tau^*$:
\begin{equation*}
    \norm{u(\tau^*)}_{L^2}\leq e^{-\mu\tau^*}\norm{u_0}_{L^2}.
\end{equation*}
Repeating this argument, we have 
\begin{equation*}
     \norm{u(n\tau^*)}_{L^2}\leq e^{-\mu n\tau^*}\norm{u_0}_{L^2}
\end{equation*}
for any $n\in\mathbb{N}$. For any $t>0$, there exists $n\in\mathbb{N}$ such that $t\in[n\tau^*, (n+1)\tau^*)$. Since $t-n\tau^*\leq\tau^*\leq T_0^2(\norm{u_0}_{L^2})$, further we have
\begin{equation*}
    \norm{u(t)}_{L^2}\leq 2e^{-\mu n\tau^*}\norm{u_0}_{L^2}\leq 2e^{-\mu t+\mu\tau^*}\norm{u_0}_{L^2}\leq \beta e^{-\mu t}\norm{u_0}_{L^2},
\end{equation*}
where $\beta=2e^{\frac{1}{10}}$.
\end{proof}
Now, we remain to prove these lemmas.
\begin{proof}[Proof of Lemma \ref{lem: L2 unif bdd}]
Let $B\defeq\norm{u(t_0)}_{L^2}$. In Proposition~\ref{prop:adv mild equi with weak} we showed that the mild solution is also a weak solution, and it satisfies the energy identity~\eqref{e:energy identity for adv} on $[0, T]$. Taking derivative with respect to $t$ on the both hand sides of identity~\eqref{e:energy identity for adv}, we get
\begin{equation*}
    \norm{u(t)}_{L^2} \cdot \frac{d}{dt}\norm{u(t)}_{L^2}=-\norm{\lap u(t)}_{L^2}^2 + \norm{\grad u(t)}_{L^p}^p.
\end{equation*}
By Gagliardo-Nirenberg's inequality and Young's inequality with $\epsilon$, we further get
\begin{equation}
\label{e:energy estimate}
    \norm{u}_{L^2} \cdot \frac{d}{dt}\norm{u}_{L^2}
    \leq-\frac{1}{2}\norm{\lap u}_{L^2}^2 + A_p\norm{u}_{L^2}^{\tfrac{2}{3-p}}
    \leqslant A_p\norm{u}_{L^2}^{\tfrac{2}{3-p}},
\end{equation}
where $A_p$ is a constant that only depends on $p$.
Thus, for a.e. $t\in(0,T)$ we have
\begin{equation*}
    \frac{d}{dt}\int_{B}^{\norm{u(t)}_{L^2}}\,\frac{y}{A_p y^{\tfrac{2}{3-p}}}\,dy=\frac{\norm{u}_{L^2}\frac{d}{dt}\norm{u}_{L^2}}{A_p\norm{u}_{L^2}^{\frac{2}{3-p}}}\leq 1,
\end{equation*}
which further implies for all $t\in[t_0,t_0+T_0^{2}(B)]$, we have
\begin{equation*}
    \int_{B}^{\norm{u(t)}_{L^2}}\,\frac{y}{A_p y^{\tfrac{2}{3-p}}}\,dy
    \leq t-t_0\leq T_0^2(B)
    = \int_{B}^{2B}\,\frac{y}{A_p y^{\tfrac{2}{3-p}}}\,dy.
\end{equation*}
Since the integrands are strictly positive and identical, which forces $\norm{u(t)}_{L^2}\leq 2B$ in $[t_0,t_0+T_0^{2}(B)]$ and the proof is complete. 
\end{proof}

\begin{proof}[Proof of Lemma \ref{lem:mean large H2}]
By Lemma~\ref{lem: L2 unif bdd}, the $L^2$ norm of $u$ is uniformly bounded by $2\norm{u(t_0)}_{L^2}$ on the interval $[t_0,t_0+\tau]$. Integrating~\eqref{e:energy estimate} from $t_0$ to $t_0+\tau$ and apply the assumption~\eqref{e:mean large of H2}, we get
\begin{align*}
    \norm{u(t_0+\tau)}_{L^2}^2
    &\leq\norm{u(t_0)}_{L^2}^2-\int_{t_0}^{t_0+\tau}\,\norm{\lap u(t)}_{L^2}^2 dt + 2\int_{t_0}^{t_0+\tau}A_p\norm{u(t)}_{L^2}^{\tfrac{2}{3-p}}\,dt\\
    &\leq\norm{u(t_0)}_{L^2}^2-\int_{t_0}^{t_0+\tau}\,\norm{\lap u(t)}_{L^2}^2\,dt + 2\tau A_p\left(2\norm{u(t_0)}_{L^2}\right)^{\tfrac{2}{3-p}}\\
    &\leq(1-2\mu\tau)\norm{u(t_0)}_{L^2}^2\leq e^{-2\mu\tau}\norm{u(t_0)}_{L^2}^2,
\end{align*}
Finally, we take square root on the both sides of the inequality to complete the proof.
\end{proof}
In order to simplify the proof of Lemma~\ref{lem: mean small H2}, we introduce and prove following lemma. 
\begin{lemma}\label{lem:diff btw hdq and ahdq}
For any $f\in L^1$, $v\in L^\infty([0,\infty);L^{\infty})$, consider:
\begin{itemize}
    \item [(1).] Hyper-diffusion equation with initial data $f$:
    \begin{equation}
    \label{e:HDE}
        \theta_t + \lap^2 \theta = 0\quad\text{with}\quad \theta_0=f;
    \end{equation}
    \item [(2).] Advective hyper-diffusion equation with initial data $f$:
    \begin{equation}
    \label{e:AHDE}
        \phi_t + v\cdot\grad\phi + \lap^2 \phi = 0\quad\text{with}\quad \phi_0=f.
    \end{equation}
\end{itemize}
Then 
\begin{equation}
    \norm{e^{-t\FL}f-S_{0, t}f}_{L^2}
    =\norm{\theta(t,\cdot)-\phi(t,\cdot)}_{L^2}\leqslant C t^{1/4}\norm{v}_{L^\infty([0,\infty);L^{\infty})}\norm{f}_{L^1}.
\end{equation}
\begin{proof}
Let $w=\theta-\phi$, subtracting ~\eqref{e:HDE} with~\eqref{e:AHDE} yields
\begin{equation}
    w_t+\lap^2 w+v\cdot\nabla w=v\cdot\nabla\theta
    \quad\text{with}\quad w(0)=0.
\end{equation}
Multiply $w$ on the both to get following energy identity:
\begin{equation}
    \frac{d}{dt}\norm{w(t)}_{L^2}^2+2\norm{\lap w(t)}_{L^2}^2=2\int_{\T^2} v\cdot\nabla\theta\cdot w dx
\end{equation}
Since $v$ is divergence free, we have
\begin{align*}
    \frac{d}{dt}\norm{w(t)}_{L^2}^2+2\norm{\lap w(t)}_{L^2}^2
    &=2\int_{\T^2} \theta v\cdot\nabla w dx\\
    &\leqslant 2\norm{v}_{L^\infty([0,\infty);L^{\infty})}\norm{\theta}_{L^2}\norm{\grad w}_{L^2}\\
    &\leqslant C\norm{v}_{L^\infty([0,\infty);L^{\infty})}^2\norm{\theta}_{L^2}^2+\norm{\lap w}_{L^2}^2,
\end{align*}
where we applied Poincare and Young's inequalities in the last step. Therefore,
\begin{equation*}
    \frac{d}{dt}\norm{w}_{L^2}^2\leqslant C\norm{v}_{L^\infty([0,\infty);L^{\infty})}^2\norm{\theta}_{L^2}^2.
\end{equation*}
Integrating the above inequality from $0$ to $t$ on the both sides gives
\begin{align*}
    \norm{w(t)}_{L^2}^2
    &\leqslant  C\norm{v}_{L^\infty([0,\infty);L^{\infty})}^2\int_0^t\norm{e^{-s\FL}f}_{L^2}^2 ds\\
    &\leqslant C\norm{v}_{L^\infty([0,\infty);L^{\infty})}^2\int_0^t\norm{e^{-s\FL}}_{L^1\rightarrow L^2}^2\norm{f}_{L^1}^2 ds\\
    &\leqslant C\norm{v}_{L^\infty([0,\infty);L^{\infty})}^2\norm{f}_{L^1}^2\int_0^t s^{-\frac{1}{2}}ds\\
    &\leqslant Ct^{\frac{1}{2}}\norm{v}_{L^\infty([0,\infty);L^{\infty})}^2 \norm{f}_{L^1}^2
\end{align*}
Finally, take square root on the both to complete the proof.
\end{proof}
\end{lemma}
Now we are ready to prove Lemma~\ref{lem: mean small H2}.
\begin{proof}[Proof of Lemma~\ref{lem: mean small H2}]
For simplicity, let $B$ denote $\norm{u(t_0)}_{L^2}$. 
Since $u$ is a mild solution of $\eqref{e:adv 4 order eqn}$, and by Remark~\ref{rmk:redefine N} we can write $u(t_0+\tau^*)$ in the following Duhamel's form:
\begin{equation*}
u(t_0+\tau^*)=\mathcal S_{t_0, t_0+\tau^*}u(t_0)\,-\,\int_{t_0}^{t_0+\tau^*}\,\mathcal S_{s, t_0+\tau^*}N(u(s))\,ds.
\end{equation*}
Take the $L^2$ norm on the both hand sides to get:
\begin{align*}
&\norm{u(t_0+\tau^*)}_{L^2}\\
& \quad \leqslant\norm{\mathcal S_{t_0, t_0+\tau^*}u(t_0)}_{L^2}\,+\,\int_{t_0}^{t_0+\tau^*}\,\norm{\mathcal S_{s, t_0+\tau^*}N(u(s))}_{L^2}\,ds\\
&\quad \leqslant \norm{\mathcal S_{t_0, t_0+\tau^*}u(t_0)}_{L^2}\,+\,\int_{t_0}^{t_0+\tau^*}\, \left\|\left(\mathcal S_{s, t_0+\tau^*}-e^{(t_0+\tau^*-s)}\right)N(u(s))\right\|_{L^2}\,ds \\ &\quad \quad +\int_{t_0}^{t_0+\tau^*}\,\norm{e^{(t_0+\tau^*-s)}N(u(s))}_{L^2}\,ds\\
& \quad \defeq I_1+I_2+I_3.
\end{align*}

\medskip

\textit{Estimate of $I_1$.}  By the definition of dissipation time, for the term $\norm{\mathcal S_{t_0, t_0+\tau^*}u(t_0)}_{L^2}$ we have:
    \begin{equation*}
        \norm{\mathcal S_{t_0, t_0+\tau^*}u(t_0)}_{L^2}\leqslant\frac{1}{2}\norm{u(t_0)}_{L^2}=\frac{B}{2}.
    \end{equation*}
    
\medskip

\textit{Estimate of $I_2$.}  Firstly, recall that for a.e. $s>0$ we have $N(u(s))\in L^1(\T^d)$, and 
\begin{equation*}
    \norm{N(u)}_{L^1}\leqslant C\norm{\grad u}_{L^{2(p-2)}}^{p-2}\norm{\lap u}_{L^2}\leqslant C\norm{\lap u}_{L^2}^{\frac{p}{2}}\norm{u}_{L^2}^{\frac{p-2}{2}}.
\end{equation*} 
Therefore, for the second term above we can apply Lemma~\ref{lem:diff btw hdq and ahdq} to get:
\begin{eqnarray*}
&&\int_{t_0}^{t_0+\tau^*}\, \left\|\left(\mathcal S_{s, t_0+\tau^*}-e^{(t_0+\tau^*-s)}\right)N(u(s))\right\|_{L^2}\,ds\\
&& \quad \leqslant C\norm{v}_{L^\infty([0,\infty);L^{\infty})}\int_{t_0}^{t_0+\tau^*}\left(t_0+\tau^*-s\right)^{\frac{1}{4}}\norm{N(u)}_{L^1}ds\\
&& \quad \leqslant C\norm{v}_{L^\infty([0,\infty);L^{\infty})}\int_{t_0}^{t_0+\tau^*}\left(t_0+\tau^*-s\right)^{\frac{1}{4}}\norm{\lap u}_{L^2}^{\frac{p}{2}}\norm{ u}_{L^2}^{\frac{p-2}{2}}ds\\
&& \quad \leqslant
C\norm{v}_{L^\infty([0,\infty);L^{\infty})}\left(2B\right)^{\frac{p-2}{2}}\left(\int_{t_0}^{t_0+\tau^*}\left(t_0+\tau^*-s\right)^{\frac{1}{4-p}}ds\right)^{\frac{4-p}{4}}\left(\int_{t_0}^{t_0+\tau^*}\norm{\lap u(s)}_{L^2}^2 ds\right)^{\frac{p}{4}}\\
&& \quad \leqslant
C\norm{v}_{L^\infty([0,\infty);L^{\infty})}\left(2B\right)^{\frac{p-2}{2}}(\tau^*)^{\frac{5-p}{4}}\left(\tau^*\left(2A_p(2B)^{\frac{2}{3-p}}+2\mu B^2\right)\right)^{\frac{p}{4}}\\
&& \quad \leqslant
C\norm{v}_{L^\infty([0,\infty);L^{\infty})}(\tau^*)^{\frac{5}{4}}B^{p-1}\left(A_p B^{\frac{2(p-2)}{3-p}}+\mu\right)^{\frac{p}{4}}
\end{eqnarray*}
    
\medskip

\textit{Estimate of $I_3$.}  Similarly, for the last term. By applying Lemma~\ref{lem: semi group on divergence} we have
\begin{align*}
&\int_{t_0}^{t_0+\tau^*}\,\norm{e^{(t_0+\tau^*-s)}N(u(s))}_{L^2}\,ds\\
& \quad \leqslant \int_{t_0}^{t_0+\tau^*}\,\left(t_0+\tau^*-s\right)^{-\frac{1}{4}}\norm{N(u)
        }_{L^1}\,ds\\
& \quad \leqslant\int_{t_0}^{t_0+\tau^*}\left(t_0+\tau^*-s\right)^{-\frac{1}{4}}\norm{\lap u}_{L^2}^{\frac{p}{2}}\norm{ u}_{L^2}^{\frac{p-2}{2}}ds\\
& \quad \leqslant
        \left(2B\right)^{\frac{p-2}{2}}\left(\int_{t_0}^{t_0+\tau^*}\left(t_0+\tau^*-s\right)^{-\frac{1}{4-p}}ds\right)^{\frac{4-p}{4}}\left(\int_{t_0}^{t_0+\tau^*}\norm{\lap u(s)}_{L^2}^2 ds\right)^{\frac{p}{4}}\\
& \quad \leqslant
(\tau^*)^{\frac{3}{4}}B^{p-1}\left(A_p B^{\frac{2(p-2)}{3-p}}+\mu\right)^{\frac{p}{4}}
\end{align*}

Finally, combine the estimations above and the assumption \eqref{e:dependence of T1}, we obtain
\begin{align*}
&\norm{u(t_0+\tau^*)}_{L^2}\\
& \quad\leqslant \frac{B}{2}+  C\left(\norm{v}_{L^\infty([0,\infty);L^{\infty})}(\tau^*)^{\frac{5}{4}}+(\tau^*)^{\frac{3}{4}}\right)\left(A_p B^{\frac{2(p-2)}{3-p}}+\mu\right)^{\frac{p}{4}}B^{p-1}\\
& \quad \leqslant (1-\mu\tau^*)B\leqslant e^{-\mu\tau^*}B.
\end{align*}
The proof is complete. 
\end{proof}

\section{Finite time blow-up of the solutions}
\label{sec:final sec}
The goal of this section is to study the blow-up behavior corresponding to equation~\eqref{e:4 order eqn}, in particular, we prove parallel results of \cite[Theorem 1.1, Theorem 1.2]{ishige2020blowup} on the $2$-torus. Moreover, instead of using Picard iteration as in \cite{ishige2020blowup} we utilize Banach contraction mapping theorem to construct solutions, which makes the proof much easier. On the other hand, it is worth to mention that the proof in \cite{ishige2020blowup} makes use of the dilation structure of~\eqref{e:4 order eqn}, which does not hold in our setting. 

Here is the plan of this section: in Section~\ref{sec:revisit semigroup},  we first recall one more semigroup estimate, which will play an essential role of our construction. Next, in Section~\ref{sec:H1 mild soln}, we construct mild solutions induced by an $W^{1,\infty}$ initial data. In Section~\ref{sec:blowup with H2 data}, we first show that the mild solutions constructed in Section~\ref{sec:H1 mild soln} possess better regularity, if the initial data belongs to $H^2\cap W^{1,\infty}$. Then we provide non-trivial $H^2\cap W^{1,\infty}$ initial data such that the corresponding mild solution must blow up in a finite time. Finally, in Section~\ref{sec:characterize blowup}, we show that in the previous example of blow-up, the $L^2$ norm of the solution must blow up, moreover, we characterize the blow-up behaviour by providing a quantitative blow-up rate of the $L^2$ norm.
\subsection{Revisit of Fourier analysis on torus.}
\label{sec:revisit semigroup}
A main ingredient in the proof of the main theorem of this section is to estimate the $L^r$-$L^{\infty}$-norm of $e^{-t\FL}$. Recall for any $f \in L^p(\T^2)$, $p \ge 2$, $\hat{f}=\left\{ \hat{f}(k) \right\}_{k \in \Z^2} \in \ell^2$ and for any $\{a_n\}_{n \in \Z^2} \in \ell^2$ and $x \in \T^2$, $\mathcal F^{-1} \left( \left( a_n \right)_{k \in \Z^2} \right)(x)=\sum\limits_{k \in \Z^2} a_k e^{2\pi i k \cdot x} \in L^2(\T^2)$.

\begin{lemma}
\label{lem:L_infty estimate}
Let $1\leqslant r\leqslant\infty$ and $j=\{0,1,2,\cdots\}$,  then for any $f\in L^{r}(\T^2)$ we have
\begin{equation}
\norm{\grad^{j}e^{-t\FL} f}_{L^{\infty}}\leqslant C_{j}t^{-\tfrac{1}{2r}-\tfrac{j}{4}}\norm{f}_{L^{r}},
\end{equation}
for some $C_j>0$.
\end{lemma}
\begin{proof}
For any $x\in \T^2$,
\begin{equation}
\label{e:Fourier identity}
    \left| \grad^{j}e^{-t\FL} f(x)\right|=\left| \mathcal F^{-1} \widehat{\left( \grad^{j}e^{-t\FL} f \right)}(x) \right|
    =\left|\sum_{\k\in\Z^2}\left|\k\right|^{j}e^{-t|\k|^4} \hat{f}(k)e^{2\pi i\k\cdot x}\right|
\end{equation}
Thus for $r=1$, we have
\begin{eqnarray*}
    \left| \grad^{j}e^{-t\FL} f(x)\right |\ 
    &\leq& \norm{f}_{L^1}\cdot \sum_{\k\in\Z^2} \left|\k\right|^j   
      e^{-t|\k|^4}\\
    &\leq& \norm{f}_{L^1}\cdot\int_{\R^2}\left|x\right|^{j} e^{-t|
      x|^4}dx\\
    &\leq& t^{-\tfrac{1}{2}-\tfrac{j}{4}}\norm{f}_{L^1}.
\end{eqnarray*}
On the other hand, for $r=\infty$, 
\begin{eqnarray*}
    \left| \grad^{j}e^{-t\FL} f(x)\right |\ 
    &\leq& \sup_{x\in\R^2}\left|x\right|^{j} e^{-\tfrac{t|x|^4}{2}}
    \left|\sum_{\k\in Z^2}\left|\hat{f}(\k)\right|e^{-\tfrac{t|\k|^4}{2}}
    \right|\\
    &\leq& C t^{-\tfrac{j}{4}}\norm{f}_{L^{\infty}}.
\end{eqnarray*}
Finally, for $1<r<\infty$, the desired estimate follows form real interpolation.
\end{proof}
\subsection{Mild solutions with $W^{1,\infty}$ initial data.}
\label{sec:H1 mild soln}
We now turn to the construction of the $L^2$ blow up example. 
In contrast to the mild solutions defined in Definition~\ref{def:mild and weak soln}, we begin with constructing  mild solutions induced by $W^{1,\infty}$ initial data. We first update the definition of mild solution to the equation \eqref{e:4 order eqn} with an $W^{1,\infty}$ initial data. 

\begin{definition} \label{def:mildinitialH2}
For $2<p<3$, a function $ u\in C((0,T]; W^{1,\infty}(\T^2)),0<T$ is called a \emph{mild solution} of~\eqref{e:4 order eqn} on $(0,T]$ with initial data $ u_0\in W^{1,\infty}(\T^2)$, if for any $0 < t \leq T$, 
\begin{equation}
\label{e:volterra integral}
     u(t)=\mathcal{T}(u)(t)\defeq e^{-t\FL} u_0-\int_0^t\, \grad e^{-(t-s)\FL}\Big(\abs{\grad u}^{p-2}\grad u\Big)\,ds
\end{equation}
holds pointwisely in time with values in $W^{1,\infty}$, where the integral is defined in the B\"{o}chner sense. 
\end{definition}

\begin{remark}
Note that we have removed the endpoint $0$ in the above definition for mild solutions. It turns out when the initial data $u_0 \in W^{1, \infty}(\T^2)$, this does not quite distinguish the mild solutions if we require $u \in C\left([0, T]; W^{1, \infty}(\T^2) \right)$ (see, Theorem \ref{20200705thm01}); however, it will play a major role if the initial data merely belongs to $L^2(\T^2)$ (see, Proposition \ref{thm:local exits with small L2}), which leads to the desired $L^2$ blow-up example (see, Theorem \ref{L2blowup-rate}). 
\end{remark}

We begin with the following local existence result. 

\begin{theorem} \label{20200705thm01}
Let $u_0 \in W^{1,\infty}(\T^2)$ and $2<p<3$. Then there exists $0< \breve{T} \leq 1$, depending only on $\|u_0\|_{W^{1,\infty}}$ and $p$ such that \eqref{e:4 order eqn} admits a mild solution $\breve{u}$ on $(0, \breve{T}]$ (in the sense of Definition~\ref{def:mildinitialH2}), which is unique in $C\left([0,\breve{T}]; W^{1,\infty}(\T^2)\right)$ and hence in $C\left((0,\breve{T}]; W^{1,\infty}(\T^2)\right)$.
\end{theorem}

\begin{proof}
The proof of Theorem \ref{20200705thm01} is similar to the one of Theorem \ref{thm:local mild soln}. We only sketch the proof here and would like to leave the details to the interested reader. 

First of all, we show that for any $0<T \leqslant 1$, the map $\calT$ is bounded on $C\left([0, T]; W^{1,\infty}(\T^2)\right)$. More precisely, by using the semigroup estimate given in Lemma~\ref{lem:L_infty estimate} we have the following quantitative bound: there exists a constant $C_3>0$, such that for any $0<T \leq 1$ and $u \in C\left([0,T]; W^{1,\infty}(\T^2)\right)$, 
\begin{equation} \label{20210705eq01}
\left\|\calT(u) \right\|_{C\left([0, T]; W^{1,\infty}\right)} \le C_3 \left(\left\|u_0\right\|_{W^{1,\infty}}+T^{\tfrac{1}{2}}  \left\|u \right\|_{C\left([0, T]; W^{1,\infty}\right)}^{p-1} \right), 
\end{equation}
(this is the counterpart of Lemma \ref{lem:bdd of operator}) which follows from the following estimates:
\begin{enumerate}
    \item [(1).] $\sup\limits_{0<t \le T} \|\nabla \calT(u) \|_{L^{\infty}} \leq C_3 \left( \|u_0\|_{ W^{1,\infty}}+T^{\frac{1}{2}} \|u\|_{C\left([0, T]; W^{1,\infty}\right)}^{p-1} \right)$;

    \medskip
    
    \item [(2).] $\sup\limits_{0<t \le T} \|\calT(u) \|_{L^{\infty}} \leq C_3 \left( \|u_0\|_{ W^{1,\infty}}+T^{\frac{3}{4}} \|u\|_{C\left([0, T]; W^{1,\infty}\right)}^{p-1} \right)$; 
\end{enumerate}

\medskip

Next, we show that $\calT$ is a Lipschitz mapping on $C\left([0, T]; W^{1,\infty}(\T^2)\right)$ for any $0<T \le 1$, that is: there exists a constant $C_4>0$, such that for any $0<T \le 1$, 
\begin{align}\label{20210705eq02}
&\left\|\calT(u_1)-\calT(u_2) \right\|_{C\left([0, T]; W^{1,\infty}\right)}  \\
&\leq C_4 T^{\frac{1}{2}} \left( \|u_1\|_{C\left([0, T]; W^{1,\infty}\right)}^{p-2}+\|u_2\|_{C\left([0, T]; W^{1,\infty}\right)}^{p-2} \right)  \left\|u_1-u_2 \right\|_{C\left([0, T]; W^{1,\infty}\right)} \notag, 
\end{align}
(this is the counterpart of Lemma \ref{lem:L-cts of operator}), which, similarly, can be achieved by showing the following estimates
\begin{enumerate}
     \item [(3).] $\sup\limits_{0<t \le T} 
     \left\|\nabla \calT(u_1)-\nabla \calT(u_2) \right\|_{L^{\infty}} \\
     \leq C_4 T^{\frac{1}{2}} \left( \|u_1\|_{C\left([0, T]; W^{1,\infty}\right)}^{p-2}+\|u_2\|_{C\left([0, T]; W^{1,\infty}\right)}^{p-2} \right)  \left\|u_1-u_2 \right\|_{C\left([0, T]; W^{1,\infty}\right)} $; 
     
     \medskip
     
      \item [(4).] $\sup\limits_{0<t \le T} 
     \left\|\calT(u_1)-\calT(u_2) \right\|_{L^{\infty}} \\
     \leq C_4 T^{\frac{3}{4}} \left( \|u_1\|_{C\left([0, T]; W^{1,\infty}\right)}^{p-2}+\|u_2\|_{C\left([0, T]; W^{1,\infty}\right)}^{p-2} \right)  \left\|u_1-u_2 \right\|_{C\left([0, T]; W^{1,\infty}\right)} $.
\end{enumerate}

The final step is to apply the Banach contraction mapping theorem in a ball $\mathbb B_{\breve{R}}(0)$ in $C\left([0, T]; W^{1,\infty}\right)$, where, we can take $\breve{R} \geq 2C_0'\|u_0\|_{W^{1,\infty}}$ with $C_0'=\max\{1, C_3, C_4\}$ and $\breve{T} \leq \min \left\{1, \left(2C_0'\breve{R}^{p-2} \right)^{-2} \right\}$.
\end{proof}
\begin{remark}\label{rmk:initial estimate infty}
One may notice that in addition to the existence result, we also have the following estimate:
\begin{equation*}
\|u\|_{C([0,T],W^{1,\infty})}\leqslant 2C_0'\|u_0\|_{W^{1,\infty}},
\end{equation*}
whenever $T\in [0,\breve{T}]$.
\end{remark}
\begin{corollary}
\label{cor:H2 argument}
With the same assumptions in Theorem~\ref{20200705thm01}, if $\breve{T}^*$ is the maximal time of existence of the mild solution $\breve{u}$ (in the sense of Definition~\ref{def:mildinitialH2}), then 
\begin{equation*}
    \limsup_{t\rightarrow \breve{T}_-^*}\norm{\breve{u}(t)}_{W^{1,\infty}(\T^2)}=\infty.
\end{equation*}
Otherwise, $\breve{T}^*=\infty$.
\end{corollary}

We omit the detail of the proof of Corollary~\ref{cor:H2 argument} since it follows from an easy modification of Corollary~\ref{cor:L2 argument}. More precisely, one can easily show that if $\limsup\limits_{t\rightarrow T_-^*}\norm{\breve{u}(t)}_{W^{1,\infty}(\T^2)}<\infty$, then the corresponding mild solution can be extended by surpassing $\breve{T}^*$, which contradicts to the assumption that  $\breve{T}^*$ is the maximal time of existence of the mild solution.

\subsection{Main result.} 
\label{sec:blowup with H2 data}
In the third part of this section, we first show that the mild solutions constructed in Section~\ref{sec:H1 mild soln} possess better regularity if the initial data further belongs to $H^2\cap W^{1,\infty}$. Then our main theorem shows that there exists non-trivial initial data $u_0\in H^2\cap W^{1,\infty}$ such that the corresponding mild solution must blow up in a finite time. This extra assumption that we need is that the energy of $u_0$ is negative, which we recall is given by
$$
E(u_0)=\frac{1}{2} \left\|\lap u_0 \right\|_{L^2}^2-\frac{1}{p} \left\|\nabla u_0\right\|_{L^p}^p. 
$$

We begin by showing that the mild solutions constructed in Section~\ref{sec:H1 mild soln} own better regularity if $u_0\in H^2$ is provided.
\begin{proposition}
\label{prop:regularity improve}
Let $T>0$, and $u$ is a mild solution with respect to Definition~\ref{def:mildinitialH2} on $(0,T]$. If further $u_0\in H^2$, then in addition to Theorem~\ref{20200705thm01}, the mild solution has a better regularity:
\begin{equation}
    u\in L^{\infty}([0, T];H^2),
\end{equation}
\end{proposition}
\begin{proof}
For any $0< t\leqslant T$ we have
\begin{eqnarray*}
&&\norm{\lap u(t)}_{L^2} = \norm{\lap \calT u(t)}_{L^2}\\
&& \quad \quad \quad \leq \left\|e^{-t\FL}\lap u_0 \right\|_{L^2}+\int_0^t\left\|\left(-\lap\right)^{\tfrac{3}{2}}e^{-(t-s)\FL}\left(\vert\grad u\vert^{p-2}\grad u\right)\right\|_{L^2}ds \\
&& \quad \quad \quad \leq \norm{u_0}_{H^2}+\int_0^t\left(t-s\right)^{-\tfrac{p+1}{4}}\norm{\grad u}_{L^2}^{p-1}ds\\
&& \quad \quad \quad \leq \norm{u_0}_{H^2}+ T^{\tfrac{3-p}{4}}\norm{u}_{C\left([0, T]; W^{1,\infty}\right)}^{p-1}.
\end{eqnarray*} 
\end{proof}
\begin{remark}\label{rmk:T initial control}
This calculation together with Theorem~\ref{20200705thm01} and Remark~\ref{rmk:initial estimate infty} gives 
\begin{equation}
\label{e:unif H2 bd}
    \norm{u}_{L^{\infty}([0, \breve{T}];H^2)}
    \leqslant \norm{u_0}_{H^2}+ C\breve{T}^{\tfrac{3-p}{4}}\norm{u_0}_{ W^{1,\infty}}^{p-1}.
\end{equation}
Here $\breve{T}$ is the same one defined on Theorem~\ref{20200705thm01}.
\end{remark}

We are ready to state the main result in this section. 
\begin{theorem}
\label{thm:blow up}
Let $2<p<3$ and $u_0\in H^2(\T^2)\cap W^{1,\infty}(\T^2)$ with $E(u_0)<0$, then the corresponding mild solution to~\eqref{e:4 order eqn} (namely, the mild solution defined in Definition~\ref{def:mildinitialH2}) must blow up in finite time.
\end{theorem}
To prove this theorem, we need to consider the following approximated problem to~\eqref{e:4 order eqn}: for $\rho\in(0,1)$ and $2<p<3$, 
\begin{equation}
\label{e:approx 4 order eqn}
\begin{cases}
    \partial_t u^{\rho} + (-\lap)^2 u^{\rho} = -\grad\cdot\big((\abs{\grad u^{\rho}}^{2}+\rho)^{\frac{p-2}{2}}\grad u^{\rho}\big); \\
    \\
   u^{\rho}(0,x)= u_0(x) \in H^2(\T^2)\cap W^{1,\infty}, 
\end{cases}
\end{equation}
which can be regarded as the $L^2$-gradient flow of the energy $E_\rho$:
\begin{equation}
\label{e:approx energy}
    E_{\rho}(\phi)\defeq
    \frac{1}{2}\int_{\T^2}\,(\lap\phi)^2\,dx - \frac{1}{p}\int_{\T^2}\,\left((\abs{\grad\phi}^2+\rho)^{\frac{p}{2}}-\rho^{\frac{p}{2}}\right)\,dx.
\end{equation}
 We formulate a definition of the mild solution $u^{\rho}$ to problem~\eqref{e:approx 4 order eqn} in the same manner as in Definition~\ref{def:mildinitialH2}. The existence and uniqueness of the solution is established in the following theorem.
\begin{theorem}
\label{thm:local mild soln for approx eqn}
Let $2<p<3$ and $u_0\in H^2(\T^2)\cap W^{1,\infty}(\T^2)$ with $\norm{u_0}_{H^2}+\norm{u_0}_{W^{1,\infty}}<D$. Then there exists $0< T\leq 1$ depending on $D$ such that for any $0<\rho<1$ equation~\eqref{e:approx 4 order eqn} admits a unique mild solution $u^{\rho}$ on $[0,T]$ with the representation
$$
     u^{\rho}(t)= e^{-t\FL}u_0\,-\,\int_0^t\, \nabla e^{-(t-s)\FL}\,\left((\abs{\grad u^{\rho}}^{2}+\rho)^{\frac{p-2}{2}}\grad u^{\rho}\right)\,ds,\qquad 0\leq t\leq T.
$$
In particular, 
\begin{equation} \label{20210708eq01}
    \sup_{0<\rho<1}\norm{u^{\rho}}_{L^{\infty}([0,T];H^2)}\leq L_1 (D+D^{p-1}),
\end{equation} 
for some $L_1>0$ that is independent on $D$.
\end{theorem}

\begin{proof}
The proof of Theorem~\ref{thm:local mild soln for approx eqn} follows from a proper modification of Theorem \ref{20200705thm01}, Proposition~\ref{prop:regularity improve} and Remark~\ref{rmk:T initial control}, and hence we omit it here. Nevertheless, we would like to make a remark that the upper bound in the estimate \eqref{20210708eq01} follows from the choice of $\breve{T}$ in the Remark~\ref{rmk:T initial control}. 
\end{proof}

We now divide the proof of Theorem~\ref{thm:blow up} into following lemmas. We start with arguing that there exists a short time interval that $u^{\rho}$ converges to $u$ in $L^{\infty}([0,T];H^2(\T^2))$ as $\rho\rightarrow 0$.

\begin{lemma}
\label{lem:ST convergence}
Let $u$ and $u^{\rho}$ be the corresponding mild solutions with $u_0\in H^2(\T^2)\cap W^{1,\infty}(\T^2)$ of ~\eqref{e:4 order eqn} and ~\eqref{e:approx 4 order eqn}, respectively. If there is a $T_1>0$ such that 
\begin{equation*} 
    \max\left\{\norm{u}_{L^{\infty}\left([0,T_1];H^2\right)},\sup_{0<\rho<1}\norm{u^{\rho}}_{L^{\infty}\left([0,T_1];H^2\right)}\right\}<D
\end{equation*}
for some $D>0$, then there exists $T_2<T_1$, which only depends on $D$ and $T_1$, such that
$\norm{u-u^{\rho}}_{L^{\infty}([0,T_2];H^2)}\rightarrow 0$ as $\rho\rightarrow 0$.
\end{lemma}

\begin{proof}
For simplicity, let $F_{\rho}(\xi)\defeq (\abs{\xi}^{2}+\rho)^{\frac{p-2}{2}}\xi$. By Lemma~\ref{lem: semi group on divergence} and Lemma~\ref{lem: Hs of semigroup} with $s=3$, we have for any $0<t<T_1$, 
\begin{eqnarray*}
&& \norm{\lap u(t) - \lap u^{\rho}(t)}_{L^2} \leq \left\|\int_0^t\,\left(-\lap \right)^{\frac{3}{2}} e^{-(t-s)\FL}\big(F_{\rho}(\grad u^{\rho})-F(\grad u^{\rho})\big)\,ds \right\|_{L^2}\\ 
&& \quad \quad \quad \quad \quad \quad +\left\|\int_0^t\,\left(-\lap\right)^{\frac{3}{2}} e^{-(t-s)\FL} \big(F(\grad u^{\rho})-F(\grad u)\big)\,ds \right\|_{L^2}\\ 
&& \quad \leq C\int_0^t\,(t-s)^{-\frac{p+1}{4}}\norm{F_{\rho}(\grad u^{\rho})-F(\grad u^{\rho})}_{L^{\frac{2}{p-1}}}\,ds\\ \notag
&& \quad \quad \quad \quad \quad \quad +C\int_0^t\,(t-s)^{-\frac{p+1}{4}}\norm{F(\grad u^{\rho})-F(\grad u)}_{L^\frac{2}{p-1}}\,ds\\
&& \quad \leq C\int_0^t\,(t-s)^{-\frac{p+1}{4}}\left\|\rho^{\frac{p-2}{2}}\abs{\grad u^{\rho}}\right\|_{L^\frac{2}{p-1}}\, ds\\ \notag
&&\quad \quad \quad \quad \quad \quad +C\int_0^t\,(t-s)^{-\frac{p+1}{4}}\left\|(\abs{\grad u^{\rho}}^{p-2}+\abs{\grad u}^{p-2})\abs{\grad u^{\rho}-\grad u}\right\|_{L^\frac{2}{p-1}}\,ds\\
&& \quad \leq C\rho^{\frac{p-2}{2}}\int_0^t\,(t-s)^{-\frac{p+1}{4}}\norm{\grad u^{\rho}}_{L^2}\, ds\\ 
&& \quad \quad \quad \quad \quad \quad +C\int_0^t\,(t-s)^{-\frac{p+1}{4}}\big(\norm{\grad u^{\rho}}_{L^2}^{p-2}
+\norm{\grad u}_{L^2}^{p-2}\big)\norm{\grad u^{\rho}-\grad u}_{L^2}\,ds\\
&&\quad \leq C\Big(\rho^{\frac{p-2}{2}} Dt^{\frac{3-p}{4}}+D^{p-2}t^{\frac{3-p}{4}}\norm{u^{\rho}-u}_{L^\infty((0,T_1);H^2)}\Big).
\end{eqnarray*}
Similarly, we can show that for any $0<t<T_1$, 
$$
\left\|\nabla u(t)-\nabla u^{\rho}(t) \right\|_{L^2} \le C \left(\rho^{\frac{p-2}{2}}Dt^{\frac{4-p}{4}}+D^{p-2}t^{\frac{4-p}{4}} \|u^\rho-u \|_{L^\infty \left((0, T_1); H^2\right)} \right)
$$
and
$$
\left\|u(t)- u^{\rho}(t) \right\|_{L^2} \le C \left(\rho^{\frac{p-2}{2}}Dt^{\frac{5-p}{4}}+D^{p-2}t^{\frac{5-p}{4}} \|u^\rho-u \|_{L^\infty \left((0, T_1); H^2\right)} \right).
$$
Combining these estimates we have for any $0<t<T_1$, 
\begin{equation*}
    \norm{u - u^{\rho}}_{L^\infty((0, t);H^2)}\leq C (1+t^{\frac{1}{2}}) \Big(\rho^{\frac{p-2}{2}} Dt^{\frac{3-p}{4}}+D^{p-2} t^{\frac{3-p}{4}} \norm{u^{\rho}-u}_{L^\infty((0,T_1);H^2)}\Big).
\end{equation*}
Therefore, there exists some $T_2<T_1$ sufficiently small, which only depends on $T_1$ and $D$, such that  
$$
C(1+T_2^{\frac{1}{2}})D^{p-2} T_2^{\frac{3-p}{4}}<\frac{1}{2}, 
$$
which further gives 
\begin{equation*}
    \norm{u - u^{\rho}}_{L^\infty((0,T_2);H^2)}< C(1+T_2^{\frac{1}{2}}) \rho^{\frac{p-2}{2}} D T_2^{\frac{3-p}{4}}.
\end{equation*}
This clearly suggests that for such choice of $T_2$, one has $\norm{u - u^{\rho}}_{L^\infty((0,T_2);H^2)}\rightarrow 0$ as $\rho\rightarrow 0$.
\end{proof}

\begin{lemma}
\label{lem:local energy}
With the same assumptions in Lemma \ref{lem:ST convergence}, if further we have $E(u_0)<0$, then there exists some $0<T_3<1$ which only depends on $D$, such that for any $0<t<T_3$, the following energy estimates hold
\begin{equation}
\label{20210710eq01} 
    E(u(t))\leq E(u_0)<0,
\end{equation}
\begin{equation}
\label{20210710eq02} 
    \frac{E(u(t))}{\norm{u(t)}_{L^2}^{p}} \leqslant \frac{ E(u_0)}{\norm{u_0}_{L^2}^{p}},
\end{equation}
and
\begin{equation} \label{20210710eq03} 
    \frac{1}{\norm{u(t)}_{L^2}^{p-2}}-\frac{1}{\norm{u_0}_{L^2}^{p-2}}\leq \frac{p(p-2)E(u_0) t}{\norm{u_0}_{L^2}^{p}}. 
\end{equation}
\end{lemma}

\begin{proof}
We divide the proof of Lemma \ref{lem:local energy} into two parts. 

\medskip

\textit{Step 1: We start with considering the case when $u_0 \in C^\infty$. More precisely, we show that if $u_0 \in C^\infty$, then the three estimates \eqref{20210710eq01}, \eqref{20210710eq02} and \eqref{20210710eq03} hold.  
}

\medskip

To begin with, we note that since $u_0 \in C^\infty$,  $u^{\rho}$ is indeed a classical solution of \eqref{e:approx 4 order eqn} as long as $\|u^{\rho}\|_{L^2}<\infty$. To see this, one may refer to, for example, \cite[Chapter 2, Theorem 2.1]{KiselevXu16} for a similar argument. 

Using the above argument, we multiply $\partial_t u^{\rho}$ on both sides of ~\eqref{e:approx 4 order eqn} and then apply integration by parts (see, also \cite[Page 949, Step 3]{ishige2020blowup}) to obtain
\begin{equation}
\label{e:L2 gradient}
    0\leq\int_s^t\,\int_{\T^2}\,(\partial_t u^{\rho})^2\,dx\,d\tau\,=\,E_{\rho}(u^{\rho}(s))\,-\,E_{\rho}(u^{\rho}(t))
\end{equation}
for $0<s<t<T_1$. Here $T_1$ is the same one as in Lemma~\ref{lem:ST convergence}. On the other hand, multiplying~\eqref{e:approx 4 order eqn} by $u^{\rho}$ and integrating by parts, we have
\begin{eqnarray}
\label{e:decay of L2 with sigma}
\frac{d}{dt}\norm{u^{\rho}(t)}_{L^2}^2%
&=&-2\int_{\T^2}\,u^{\rho}\Big((-\lap)^2 u^{\rho}+\grad\cdot F_{\rho}(\grad u^{\rho})\Big)\,dx \nonumber\\
&=& -2pE_{\rho}(u^{\rho}(t))+G_\rho(u^{\rho}(t))
\end{eqnarray}
for $t\in[0,T_1]$, where 
\begin{equation*}
    G_{\rho}(\phi) \defeq (p-2)\norm{\lap\phi}_{L^2}^2-2\rho\int_{\T^2}\,\Big((\abs{\grad\phi}^2+\rho)^{\frac{p-2}{2}}-\rho^{\frac{p-2}{2}}\Big)\,dx.
\end{equation*}
Combining \eqref{e:L2 gradient} and~\eqref{e:decay of L2 with sigma}, we have
\begin{align*}
    \norm{u^{\rho}(t)}_{L^2}^2\frac{d}{dt}\big(-E_{\rho}(u^{\rho}(t))\big)
    &=\norm{u^{\rho}(t)}_{L^2}^2\norm{\partial_t u^{\rho}(t)}_{L^2}^2\\ \notag
    &\geq\frac{1}{4}\Big(\frac{d}{dt}\norm{u^{\rho}(t)}_{L^2}^2\Big)^2\\ \notag
    &=\frac{1}{4}\Big(-2p E_{\rho}(u^{\rho}(t))+G_{\rho}(u^{\rho}(t))\Big)\frac{d}{dt}\norm{u^{\rho}(t)}_{L^2}^2,
\end{align*}
which further yields 
\begin{eqnarray*}
&&\frac{d}{dt}\Big(-\frac{E_{\rho}(u^{\rho}(t))}{\norm{u^{\rho}(t)}_{L^2}^p}\Big)
=\frac{\norm{u^{\rho}}_{L^2}^2\frac{d}{dt}\left(-E_{\rho}\left(u^{\rho}\right)\right)}{\norm{u^{\rho}}_{L^2}^{p+2}}+\frac{\frac{p}{2}E_{\rho}\left(u^{\rho}\right)\frac{d}{dt}\norm{u^{\rho}}_{L^2}^2}{\norm{u^{\rho}}_{L^2}^{p+2}}\\\notag
&& \quad \quad \quad \geq\frac{\frac{1}{4} \cdot \Big(-2p E_{\rho}(u^{\rho}(t))+G_{\rho}(u^{\rho}(t))\Big)\frac{d}{dt}\norm{u^{\rho}(t)}_{L^2}^2}{\norm{u^{\rho}}_{L^2}^{p+2}}+\frac{\frac{p}{2} \cdot E_{\rho}\left(u^{\rho}\right)\frac{d}{dt}\norm{u^{\rho}}_{L^2}^2}{\norm{u^{\rho}}_{L^2}^{p+2}}\\\notag
&& \quad \quad \quad = {\frac{G_{\rho}(u^{\rho}(t))}{4\norm{u^{\rho}(t)}_{L^2}^{p+2}} \cdot \frac{d}{dt}\norm{u^{\rho}(t)}_{L^2}^2}\\ \notag
&& \quad \quad \quad =\frac{G_{\rho}(u^{\rho}(t))}{4\norm{u^{\rho}(t)}_{L^2}^{p+2}} \cdot \Big(-2p E_{\rho}(u^{\rho}(t))+G_{\rho}(u^{\rho}(t))\Big)
\end{eqnarray*}
for $t\in[0,T_1]$. Integrating both sides of above inequality from $0$ to $\tau$ along the time variable, where $\tau\in[0,T_1]$, we get
\begin{align}
\label{e:EL2 lower bdd with sigma}
    &-\frac{E_{\rho}(u^{\rho}(\tau))}{\norm{u^{\rho}(\tau)}_{L^2}^p}+\frac{E_{\rho}(u_0)}{\norm{u_0}_{L^2}^p} \notag \\ 
    & \quad \quad \quad \quad \geq \frac{1}{4} \cdot \int_0^{\tau}\,\frac{G_{\rho}(u^{\rho}(t))}{\norm{u^{\rho}(t)}_{L^2}^{p+2}} \Big(-2p E_{\rho}(u^{\rho}(t))+G_{\rho}(u^{\rho}(t))\Big)\,dt.
\end{align}
By \eqref{e:decay of L2 with sigma}, we have
\begin{align*}
    -\frac{1}{p(p-2)} \cdot \frac{d}{dt} \left(\frac{1}{\norm{u^{\rho}(t)}_{L^2}^{p-2}} \right)
    &=\frac{1}{2p \norm{u^{\rho}(t)}_{L^2}^{p}} \cdot \frac{d}{dt}\norm{u^{\rho}(t)}_{L^2}^2\\ \notag
    &=\frac{1}{2p} \cdot \frac{-2pE_{\rho}(u^{\rho}(t))+G_\rho(u^{\rho}(t))}{\norm{u^{\rho}(t)}_{L^2}^{p}}.
\end{align*}
Thus,
\begin{align}
\label{e:sigma L2 with power -p+2}
    \frac{1}{\norm{u^{\rho}(\tau)}_{L^2}^{p-2}}-\frac{1}{\norm{u_0}_{L^2}^{p-2}}
    =-\frac{p-2}{2} \cdot \int_0^{\tau} \frac{-2pE_{\rho}(u^{\rho}(t))+G_\rho(u^{\rho}(t))}{\norm{u^{\rho}(t)}_{L^2}^{p}}dt
\end{align}
for $\tau\in[0,T_1]$. 

Now we use Lemma~\ref{lem:ST convergence} to see that 
\begin{equation} \label{20210711eq01}
    E_{\rho}(u^{\rho}(t))\rightarrow E(u(t))\qquad\text{uniformly on}\quad [0,T_2]\quad\text{as}\quad\rho\rightarrow 0,
\end{equation} 
and
\begin{equation}  \label{20210711eq02}
    G_{\rho}(u^{\rho}(t))\rightarrow
    (p-2)\norm{\lap u(t)}_{L^2}^2\qquad\text{uniformly on}\quad [0,T_2]\quad\text{as}\quad\rho\rightarrow 0.
\end{equation} 

\medskip

Plugging~\eqref{20210711eq01} in~\eqref{e:L2 gradient}, we have
\begin{equation}
\label{e:decay of Energy}
    E(u(t))\leq E(u(s))\qquad\text{for}\quad 0\leq s\leq t\leq T_2,
\end{equation}
which in particular verifies \eqref{20210710eq01}. 

\medskip

While if we apply~\eqref{20210711eq02} in~\eqref{e:EL2 lower bdd with sigma}, one can check that 
\begin{align}
\label{e:EL2 lower bdd}
    &-\frac{E(u(t))}{\norm{u(t)}_{L^2}^{p}}+ \frac{E(u_0)}{\norm{u_0}_{L^2}^{p}} \notag \\ 
    & \quad \quad \geq \frac{1}{4}\int_0^t\,\frac{(p-2)\norm{\lap u(t)}_{L^2}^2}{\norm{u(t)}_{L^2}^{p+2}} \cdot \Big(-2pE(u(t))+(p-2)\norm{\lap u(t)}_{L^2}^2\Big)\,ds 
\end{align}
for $t\in [0,T_2]$. Since $E(u_0)<0$, by~\eqref{e:decay of Energy} we get $E(u(t))<0$ for $0\leq t\leq T_2$. Using this fact in~\eqref{e:EL2 lower bdd}, we see the right hand side of \eqref{e:EL2 lower bdd} is non-negative and therefore 
\begin{equation*}
    -\frac{E(u(t))}{\norm{u(t)}_{L^2}^{p}}+ \frac{E(u_0)}{\norm{u_0}_{L^2}^{p}} \geq 0
\end{equation*}
for $t\in[0,T_2]$, which is \eqref{20210710eq02}. 

\medskip

Finally, we prove \eqref{20210710eq03}. Letting $\rho \to 0$ in ~\eqref{e:sigma L2 with power -p+2}, we have  
\begin{align}
\label{e:decay of L2 with power -p+2}
    \frac{1}{\norm{u(t)}_{L^2}^{p-2}}-\frac{1}{\norm{u_0}_{L^2}^{p-2}}
    =-\frac{p-2}{2} \cdot \int_0^{t}\,\frac{-2p E(u(s))+(p-2)\norm{\lap u(s)}_{L^2}^2}{\norm{u(s)}_{L^2}^p}ds
\end{align}
for $t\in[0,T_2]$. Combining~\eqref{20210710eq02} with~\eqref{e:decay of L2 with power -p+2} yields 
\begin{equation*}
    \frac{1}{\norm{u(t)}_{L^2}^{p-2}}-\frac{1}{\norm{u_0}_{L^2}^{p-2}}\leq \frac{p(p-2)E(u_0)t}{\norm{u_0}_{L^2}^{p}}
\end{equation*}
for $t\in[0,T_2]$. The proof of \eqref{20210710eq03} is complete. 

\medskip

\textit{Step II: We now prove the desired estimates \eqref{20210710eq01}, \eqref{20210710eq02} and \eqref{20210710eq03} under the assumption when $u_0 \in H^2\cap W^{1,\infty}$.}

\medskip

Note that it suffices to prove the following: there exists some $0<T_3<1$, which only depends on $D$ such that for any $0<\varepsilon<1$ and $0<t<T_3$, the following estimates hold
\begin{equation}
\label{20210710eq04} \tag{\ref{20210710eq01}'}
    E(u(t))<\varepsilon E(u_0)<0,
\end{equation}
\begin{equation}
\label{20210710eq05} \tag{\ref{20210710eq02}'}
    \frac{E(u(t))}{\norm{u(t)}_{L^2}^{p}}<\frac{ \varepsilon E(u_0)}{\norm{u_0}_{L^2}^{p}},
\end{equation}
and
\begin{equation} \label{20210710eq06} \tag{\ref{20210710eq03}'} 
    \frac{1}{\norm{u(t)}_{L^2}^{p-2}}-\frac{1}{\norm{u_0}_{L^2}^{p-2}}<\frac{\varepsilon p(p-2)E(u_0) t}{\norm{u_0}_{L^2}^{p}}. 
\end{equation}

Indeed, if we have already proved estimates \eqref{20210710eq04}, \eqref{20210710eq05} and \eqref{20210710eq06}, then it suffices to let $\varepsilon \to 1$ and this then recovers the desired estimates that we need. Note that it is important for our analysis that the choice of $T_3$ is \emph{independent} of the choice of $\varepsilon$. 

\medskip

Let us pick a sequence of smooth function $\left\{ \varphi_m
 \right\}_{m=1}^\infty \subset C^\infty(\T^2)$ such that
\begin{equation} \label{20210711eq10}
\varphi_m \to u_0 \quad \textrm{in} \quad H^2\cap W^{1,\infty}.
\end{equation}
Note that for $2<p<3$, the Gagliardo–Nirenberg's inequality suggests $$
\left\|\nabla u\right\|_{L^q} \le C \left\|u \right\|_{H^2}, 
$$
for any $1<q<\infty$, which together with \eqref{20210711eq10} gives 
$$
E(\varphi_m) \to E(u_0) \quad \textrm{as} \quad m \to \infty. 
$$
Let $u_m$ be the mild solution to the equation \eqref{e:4 order eqn} with initial data $\varphi_m$. By \eqref{20210711eq10}, we may assume 
\begin{equation} \label{20210711eq25}
\sup_{1 \le m<+\infty} \|\varphi_m\|_{H^2\cap W^{1,\infty}} \le 2D.
\end{equation} 
Therefore, by Proposition~\ref{prop:regularity improve} and Remark~\ref{rmk:T initial control}, there exists some $0<T'_1<1$ and $L>0$ (both parameters are independent of the choice of $m$), such that
\begin{equation} \label{20210711eq24}
\sup_{1 \le m<+\infty} \|u_m\|_{L^\infty \left([0, T'_1]; H^2 \right)} \le L\left( D+D^{p-1}\right). 
\end{equation} 
Without loss of generality, we may assume $T_1=T'_1$, otherwise, we simply replace both numbers by $\min\{T'_1, T_1\}$, which only depends on $D$. 

\medskip

We have the following claim: there exists some $0<T_3 \le T_1$ and some $C>0$, such that for any $0<t<T_3$ and $m \ge 1$, 
\begin{equation} \label{20210711eq20}
\left\|u_m-u \right\|_{L^\infty\left([0, t]; H^2 \right)} \leq C \left\|\varphi_m-u_0 \right\|_{H^2}, 
\end{equation} 
which further guarantees that 
\begin{equation} \label{20210711eq22}
u_m \to u \quad \textrm{in} \quad C \left([0, T_3]; H^2 \right). 
\end{equation} 

\medskip

\textit{Proof of the claim:} First of all, we note that for any $0<t<T_1$, there exists some constant $C>0$, such that
\begin{enumerate}
\item [(1).] 
\begin{align*}
\left\|\lap u_m(t)-\lap u(t)\right\|_{L^2}%
&\leq  C \|\varphi_m-u_0\|_{H^2} \\
& \quad  +Ct^{\frac{3-p}{4}} \|u_m-u\|_{L^\infty \left([0, t]; H^2 \right)} \cdot (DL)^{p-2} \cdot \left(1+D^{p-2}\right)^{p-2}; 
    \end{align*}
\medskip
    \item [(2).] 
\begin{align*}
\left\|\nabla u_m(t)-\nabla u(t)\right\|_{L^2}%
&\leq C \|\varphi_m-u_0\|_{H^2} \\
&\quad  +Ct^{\frac{4-p}{4}} \|u_m-u\|_{L^\infty \left([0, t]; H^2 \right)} \cdot (DL)^{p-2} \cdot \left(1+D^{p-2}\right)^{p-2}; 
    \end{align*}
    
\medskip

    \item [(3).] 
\begin{align*}
\left\| u_m(t)- u(t)\right\|_{L^2}%
&\leq  C \|\varphi_m-u_0\|_{H^2} \\
& \quad  +Ct^{\frac{5-p}{4}} \|u_m-u\|_{L^\infty \left([0, t]; H^2 \right)} \cdot (DL)^{p-2} \cdot \left(1+D^{p-2}\right)^{p-2}. 
\end{align*}
\end{enumerate}

The proof for these estimates are similar to the proof of Lemma \ref{lem:L-cts of operator} (see, also \eqref{20210705eq02}), and hence we would like to leave it to the interested reader. Now combining all these estimates, we see that there exists some $C>0$, such that for any $0<t<T_1$ and $m \ge 1$, 
\begin{eqnarray} \label{20210711eq21} 
&&\|u_m-u\|_{L^\infty \left([0, t]; H^2 \right)} \le  C \left\|\varphi_m-u_0 \right\|_{H^2} \nonumber \\
&& \quad \quad \quad \quad \quad \quad \quad \quad \quad \quad  +C t^{\frac{3-p}{4}} \|u_m-u\|_{L^\infty \left([0, t]; H^2 \right)} \cdot (DL)^{p-2} \cdot \left(1+D^{p-2}\right)^{p-2}. 
\end{eqnarray}
This suggest that we can find some $T_3 \in (0, T_1)$, such that for any $0<t<T_3$, 
$$
Ct^{\frac{3-p}{4}} \cdot (DL)^{p-2} \cdot \left(1+D^{p-2}\right)^{p-2}<\frac{1}{2}, 
$$
which together with \eqref{20210711eq21}, gives the desired claim \eqref{20210711eq20}.  

\medskip

From now on, let us fix the choice of $T_3$. Note that here we might assume $T_3 \le T_2$, otherwise, we just replace $T_3$ by $T_2$. We show that with this particular choice of $T_3$, the estimates \eqref{20210710eq04}, \eqref{20210710eq05} and \eqref{20210710eq06} hold. 

\medskip

Here we only prove \eqref{20210710eq04}, while the proofs for the other two estimates are similar and hence we omit them here. Take and fix any $0<\varepsilon<1$ and $0<t<T_3$ from now on. Our goal is to show that
\begin{equation} \label{20210711eq23}
\varepsilon E(u_0)-E(u(t))>0. 
\end{equation} 
We prove this via a standard approximation argument. We make a remark here that the choice of $m$ is allowed to be depended on $\varepsilon, t, E(u_0)$ and $p$ in this case. We now turn to the detail. 

For any $m \ge 1$, we have
\begin{eqnarray*}
\textrm{LHS of \eqref{20210711eq23}}%
&=& E(u_0)-E(u(t))-(1-\varepsilon) E(u_0) \\
&=& \left(E(u_0)-E(\varphi_m) \right)- \left(E(u(t))-E(u_m(t)) \right) \\
&& +\left(E(\varphi_m)-E(u_m(t)) \right)+ \left(\varepsilon-1 \right) E(u_0). 
\end{eqnarray*}
Now by \eqref{20210711eq10} and \eqref{20210711eq20}, we can find some $M=M(t, \varepsilon,E(u_0),p)$ sufficiently large, such that for any $m>M$, 
$$
\left|E(u_0)-E(\varphi_m) \right|+ \left|E(u(t))-E(u_m(t)) \right|<\frac{(\varepsilon-1)E(u_0)}{2}.
$$
This, together with the fact that $E(\varphi_m)-E(u_m(t)) \ge 0$, suggest that when $m \ge M$, 
$$
\textrm{LHS of \eqref{20210711eq23}}>0.
$$
The proof is complete. 
\end{proof}

We are ready to prove our main result in this section. 

\begin{proof}[Proof of Theorem \ref{thm:blow up}]
Let $[0,T_{\max})$ be the maximum time interval that~\eqref{e:4 order eqn} admits a mild solution (in the sense of Definition~\ref{def:mildinitialH2}) with $u_0\in H^2\cap W^{1,\infty}$ and $E(u_0)<0$. We prove $T_{\max} \leq -\frac{\|u_0\|_{L^2}^2}{p(p-2)E(u_0)}$ by showing contradiction. Assume~\eqref{e:4 order eqn} admits a mild solution on $[0,T_{\max})$ with $T_{\max}>-\frac{\|u_0\|_{L^2}^2}{p(p-2)E(u_0)}$. This in particular means
\begin{equation}
    u\in L^\infty([0,T];H^2(\T^2)\cap W^{1,\infty}(\T^2)),
\end{equation}
for any $T \in \left(-\frac{\|u_0\|_{L^2}^2}{p(p-2)E(u_0)}, T_{\max} \right)$ by Proposition~\ref{prop:regularity improve}. Fix such a choice of $T$, we denote $D \defeq \|u\|_{L^\infty \left([0, T]; H^2\cap W^{1,\infty} \right)}$. Then by Theorem~\ref{20200705thm01} and Theorem~\ref{thm:local mild soln for approx eqn}, there exists some $T_1$ which depends on $D$, such that 
\begin{equation*}
    \max\left\{\norm{u}_{L^{\infty}([0,T_1];H^2)},\sup_{0\leq\rho\leq 1}\norm{u^{\rho}}_{L^{\infty}([0,T_1];H^2)}\right\}\leq L(D+D^{p-1})
\end{equation*}
for some $L>0$ independent of the choice of $D$. Then by Lemma~\ref{lem:local energy}, there exits $T_3>0$, which only depends on $L$, $D$, and $T_1$, such that
\begin{equation}
\label{e:local bound bwn 0 and T2}
    \frac{1}{\norm{u(t)}_{L^2}^{p-2}}-\frac{1}{\norm{u_0}_{L^2}^{p-2}}\leq \frac{p(p-2)E(u_0) t}{\norm{u_0}_{L^2}^{p}}
\end{equation}
for $t\in[0, T_3]$. Next, by the choice of $D$ and \eqref{20210710eq01}, we have $\norm{u(T_3)}_{H^2\cap W^{1,\infty}}\leqslant D$ and $E(u(T_3))<0$. This allows us to repeat the argument above by considering the initial data  $u(T_3)$ to get
\begin{equation}
\label{e:local bound bwn T2 and 2T2}
      \frac{1}{\norm{u(t+T_3)}_{L^2}^{p-2}}-\frac{1}{\norm{u(T_3)}_{L^2}^{p-2}}\leq \frac{p(p-2)E(u(T_3))t}{\norm{u(T_3)}_{L^2}^{p}}  
\end{equation}
for $t\in[0,T_3]$. Adding~\eqref{e:local bound bwn 0 and T2} and~\eqref{e:local bound bwn T2 and 2T2}, and using \eqref{20210710eq02}, we get
\begin{equation*}
    \frac{1}{\norm{u(t)}_{L^2}^{p-2}}-\frac{1}{\norm{u_0}_{L^2}^{p-2}}\leq \frac{p(p-2)E(u_0)t}{\norm{u_0}_{L^2}^{p}}
\end{equation*}
for $t\in[0,2 T_3]$. Repeating the above argument, we are able to get for any $t\in [0,T]$:
\begin{equation}
   \frac{1}{\norm{u(t)}_{L^2}^{p-2}}-\frac{1}{\norm{u_0}_{L^2}^{p-2}}\leq \frac{p(p-2)E(u_0)t}{\norm{u_0}_{L^2}^{p}}
\end{equation}
equivalently,
\begin{equation}
\label{e:whole interval bdd}
    \frac{1}{\norm{u(t)}_{L^2}^{p-2}}\leqslant\frac{\norm{u(t)}_{L^2}^2+p(p-2)E(u_0)t}{\norm{u_0}_{L^2}^{p}}.
\end{equation}
Note that since the left hand side of \eqref{e:whole interval bdd} is positive, thus the right hand side of \eqref{e:whole interval bdd} must be positive as well. Therefore, the above inequality can be further rewritten into:
\begin{equation}
\label{e:final lower bd}
    \norm{u(t)}_{L^2}\geq\Big(\frac{\norm{u_0}_{L^2}^{p}}{\norm{u_0}_{L^2}^2+p(p-2)E(u_0)t}\Big)^{\frac{1}{p-2}}.
\end{equation}
Since $u$ is a mild solution on $[0,T]$ by assumption, Corollary~\ref{cor:H2 argument} yields 
\begin{equation}
\label{e:rhs small infty}
    \Big(\frac{\norm{u_0}_{L^2}^{p}}{\norm{u_0}_{L^2}^2+p(p-2)E(u_0)t}\Big)^{\frac{1}{p-2}}<\infty
\end{equation}
for any $t\in[0,T]$. This clearly gives a contradiction as $t\rightarrow-\frac{\|u_0\|_{L^2}^2}{p(p-2)E(u_0)}$ (since $T>-\frac{\|u_0\|_{L^2}^2}{p(p-2)E(u_0)}$). The proof is complete. 
\end{proof}

\subsection{Blow-up rate of $L^2$ norm}
\label{sec:characterize blowup}
In this section, we show that in our previous example (Theorem \ref{thm:blow up}), the $L^2$-norm of the solution must blow up (see, Theorem \ref{L2blowup-rate}). This in particular gives an example which satisfies the assumption of Corollary \ref{cor:L2 argument}. Moreover, we characterize the blow-up rate of the $L^2$ norm near the maximal time of existence.  

We begin with establishing the following local existence result.
\begin{proposition}
\label{thm:local exits with small L2}
Let $2<p<3$. There exists a $\delta_{*}>0$ sufficiently small, only depending on $p$, such that if
\begin{equation}
\label{e:blowup rate1}
\varrho^{\tfrac{3-p}{2(p-2)}}\norm{u_0}_{L^2}\leqslant\delta,
\end{equation}
for some $\delta\in\left(0,\delta_{*}\right)$ and $\varrho>0$,  then there is a mild solution $u$ (in the sense of Definition~\ref{def:mildinitialH2}) to the problem~\eqref{e:4 order eqn} in $\T^2\times\left(0, 2\varrho\right]$, such that
\begin{equation}
\label{e:grad infty estimate}
\norm{\grad u(t)}_{L^{\infty}}\leqslant Ct^{-\tfrac{1}{2}}\left(\delta_{*}\varrho^{-\tfrac{3-p}{2(p-2)}}\right)
\end{equation}
for some $C>0$ which only depends on $p$.
\end{proposition}
Before we prove the above theorem, we need to define a proper Banach space so that one can apply the Banach contraction mapping theorem. Let $\tilde{Y}_{T}$ be the collection of all measurable functions on $\T^2 \times [0, T]$ with
\begin{enumerate}
\item [(1).] $\sup\limits_{0\leq t\leq T}t^{1/4}\norm{u(t)}_{L^{\infty}}<\infty$;

\item [(2).] $\sup\limits_{0\leq t\leq T}t^{1/2}\norm{\grad u(t)}_{L^{\infty}}<\infty$; 

\item [(3).] $\norm{u(0)}_{L^2}<\infty$, 
\end{enumerate}
and define:
\begin{equation*}
\left[u\right]_{\tilde{Y}_{T}}\defeq\max\left\{\sup_{0\leq t\leq T}t^{1/4}\norm{u(t)}_{L^{\infty}}, \sup_{0\leq t\leq T}t^{1/2}\norm{\grad u(t)}_{L^{\infty}}\right\}.
\end{equation*}
Then $\tilde{Y}_{T}$ is a Banach space equipping with the following norm:
\begin{equation}
\norm{u}_{\tilde{Y}_{T}}\defeq\norm{u(0)}_{L^2}+ \left[u\right]_{\tilde{Y}_{T}}.
\end{equation}
\begin{proof}[Proof of Proposition~\ref{thm:local exits with small L2}]
The proof of Proposition~\ref{thm:local exits with small L2} is similar to the one of Theorem \ref{thm:local mild soln} and Theorem \ref{20200705thm01}. Therefore, we only sketch the proof here and omit the details. 

First of all, we show that for any $T>0$, the map $\calT$ is bounded on $\tilde{Y}_{T}$. More precisely, we have the following quantitative bound: there exists a constant $C_5>0$, such that for any $T>0$ and $u \in \tilde{Y}_{T}$, 
\begin{equation} 
\left\|\calT(u) \right\|_{\tilde{Y}_{T}} \le C_5 \left(\left\|u_0\right\|_{L^2}+T^{\tfrac{3-p}{2}}  \left\|u \right\|_{\tilde{Y}_{T}}^{p-1} \right), 
\end{equation}
which follows from the following estimates:
\begin{enumerate}
    \item [(1).] $\sup\limits_{0<t \le T} t^{\tfrac{1}{4}}\|\calT(u) \|_{L^{\infty}} \leq C_5 \left( \|u_0\|_{L^2}+T^{\frac{3-p}{2}} \|u\|_{\tilde{Y}_{T}}^{p-1} \right)$;

    \medskip
    
    \item [(2).] $\sup\limits_{0<t \le T} t^{\tfrac{1}{2}}\|\grad\calT(u) \|_{L^{\infty}} \leq C_5 \left( \|u_0\|_{L^2}+T^{\frac{3-p}{2}} \|u\|_{\tilde{Y}_{T}}^{p-1} \right)$; 
\end{enumerate}

\medskip

Next, we show that $\calT$ is a Lipschitz mapping on $\tilde{Y}_{T}$ for any $T>0$, that is: there exists a constant $C_6>0$, such that for any $T>0$, 
\begin{align}
\left\|\calT(u_1)-\calT(u_2) \right\|_{\tilde{Y}_{T}}  
\leq C_6 T^{\frac{3-p}{2}} \left( \|u_1\|_{\tilde{Y}_{T}}^{p-2}+\|u_2\|_{\tilde{Y}_{T}}^{p-2} \right)  \left\|u_1-u_2 \right\|_{\tilde{Y}_{T}} \notag, 
\end{align}
which, similarly, can be achieved by showing the following estimates
\begin{enumerate}
     \item [(3).] $\sup\limits_{0<t \le T} t^{\tfrac{1}{4}}
     \left\|\calT(u_1)-\calT(u_2) \right\|_{L^{\infty}} 
     \leq C_6 T^{\frac{3-p}{2}} \left( \|u_1\|_{\tilde{Y}_{T}}^{p-2}+\|u_2\|_{\tilde{Y}_{T}}^{p-2} \right)  \left\|u_1-u_2 \right\|_{\tilde{Y}_{T}} $; 
     
     \medskip
     
      \item [(4).] $\sup\limits_{0<t \le T} t^{\tfrac{1}{2}}
     \left\|\grad\calT(u_1)-\nabla \calT(u_2) \right\|_{L^{\infty}} 
     \leq C_6 T^{\frac{3-p}{2}} \left( \|u_1\|_{\tilde{Y}_{T}}^{p-2}+\|u_2\|_{\tilde{Y}_{T}}^{p-2} \right)  \left\|u_1-u_2 \right\|_{\tilde{Y}_{T}} $.
\end{enumerate}

Finally, we apply the Banach contraction mapping theorem in a ball $\mathbb B_{R}(0)$ in $\tilde{Y}_{\tilde{T}}$, where, we can take $R \geq 2C_0\|u_0\|_{L^2}$ with $C_0=\max\{1, C_5, C_6\}$ and $\tilde{T} \leq \left(4C_0 R^{p-2}\right)^{-\tfrac{2}{3-p}}$. By assumption~\eqref{e:blowup rate1}, we may take $R=2C_0\delta\varrho^{-\tfrac{3-p}{2(p-2)}}$ for instance, and this means 
\begin{equation*}
\tilde{T}\leqslant \left(4C_0\right)^{-\tfrac{2}{3-p}}\cdot\left(2C_0\delta\right)^{-\tfrac{2(p-2)}{3-p}}\cdot\varrho.
\end{equation*}
Now we take $\delta_{*}$ small enough, such that $\left(4C_0\right)^{-\tfrac{2}{3-p}}\cdot\left(2C_0\delta\right)^{-\tfrac{2(p-2)}{3-p}}>2$ since $2<p<3$. This allows us to take $\tilde{T}=2\varrho$. Finally, the estimate~\eqref{e:grad infty estimate} follows from the choice of $\tilde{T}$, $R$ and $\norm{u_0}_{L^2}$.
\end{proof}

Finally, we show that the $L^2$-norm of the solution constructed in Theorem~\ref{thm:blow up} must blow up. 

\begin{theorem}
\label{L2blowup-rate}
Under the same assumption of Theorem~\ref{thm:blow up}, and  let $T_{\max}$ be the maximal time of existence of the mild solution to~\eqref{e:4 order eqn}, then the $L^2$ norm of the solution must blow up at $T_{\max}$. Moreover, we have following quantitative blow-up rate:
\begin{equation}
\label{e:blowup rate}
\liminf_{t\nearrow T_{\max}}\left(T_{\max}-t\right)^{\tfrac{3-p}{2(p-2)}}\norm{u(t)}_{L^2}>0.
\end{equation}
\end{theorem}
\begin{proof}
It suffices to show~\eqref{e:blowup rate}, and we prove it by showing contradiction. Assume~\eqref{e:blowup rate} does not hold. Then there exists $C_{*}>0$ small enough, such that
\begin{equation}
\norm{u(t_{*})}_{L^2}\leq C_{*}\left(T_{\max}-t_{*}\right)^{-\tfrac{3-p}{2(p-2)}}
\end{equation}
for some $t_{*} \in (0, T_{\max})$.  Here we let $C_{*}\leqslant\frac{\delta_{*}}{2}$. Now let $\tilde{u}$ be the solution to~\eqref{e:4 order eqn}  with initial data $u(t_{*})\in H^2\cap W^{1,\infty}$ (this is guaranteed by Theorem~\ref{20200705thm01} and Proposition~\ref{prop:regularity improve}). By Theorem~\ref{thm:blow up}, we see that the maximal existence time $\tilde{T}_{\max}$ of $\tilde{u}$ satisfies $\tilde{T}_{\max} \le T_{\max}-t_{*}$. However, Theorem~\ref{thm:local exits with small L2} indicates that
\begin{equation*}
\tilde{T}_{\max} \ge 2\left(T_{\max}-t_{*}\right)>T_{\max}-t_{*}.
\end{equation*}
This is a contradiction.
\end{proof}

\bibliographystyle{plain}
\bibliography{refs,preprints}

\end{document}